\documentclass[a4paper,11pt,reqno]{article}
\usepackage[T1]{fontenc}
\usepackage{courier}
\usepackage[scaled=0.92]{helvet}
\usepackage{amscd,amsmath,amssymb,amsthm,amsfonts,epsfig,graphics,accents}

\usepackage[percent]{overpic} 				
\usepackage[english]{babel}
\usepackage{algpseudocode,algorithm,float}
\usepackage{amsthm,amsfonts,amssymb,amsmath,subfigure}
\usepackage{hyperref}
\hypersetup{
colorlinks=true,
citecolor=red,
linkbordercolor={1 1 1},
citebordercolor={1 1 1},
}
\usepackage{cleveref}
\usepackage{mathrsfs}
\usepackage{color}
\usepackage{bm}
\usepackage{enumitem}
\usepackage{algorithm}
\usepackage{algpseudocode}
\usepackage{float}
\usepackage[most]{tcolorbox} 
\usepackage{graphicx}
\usepackage{float}

\usepackage[scale=2]{draftwatermark}
\SetWatermarkLightness{1} 
\SetWatermarkAngle{45}
\SetWatermarkText{DRAFT}

\usepackage{xcolor} 

\newcommand{\mx}{\mbox}
\newcommand{\rw}{\rightarrow}

\newcommand{\ml}{\mathcal}

\newcommand{\pl}{\partial}

\newcommand{\beq}[1]{\begin{equation} \label{#1}}
\newcommand{\eeq}{\end{equation}}
\newcommand{\beqar}{\[ \begin{array}{rcl}}
\newcommand{\eeqar}{\end{array} \]}

\providecommand{\CC}{\mathbb{C}}

\newtheorem{satz}{Theorem}[section]
\newtheorem{prop}{Proposition}[section]

\newtheorem{lem}[satz]{Lemma}

\newtheorem{defn}[satz]{Definition}
\newtheorem{rem}{Remark}[section]

\newtheorem{ex}[satz]{Example}

\usepackage{xcolor}

\usepackage{xcolor}

\providecommand{\CC}{\mathbb{C}}

\newcommand{\snorm}[2]{\left| #1\right|_{#2}}
\newcommand{\norm}[2]{\left \lVert#1 \right\rVert_{#2}}

\usepackage[utf8]{inputenc}

\hypersetup{pdfauthor={Fortunati, A. and Pacella F.},%
            pdftitle={A Constructive Approach  ...},%
            pdfsubject={},%
            pdfkeywords={}%
}

\makeatletter
\renewcommand*{\@fnsymbol}[1]{\ensuremath{\ifcase#1\or *\or \mathsection \or \dagger \or \else \fi}}
\makeatother

\title{\LARGE{\textbf{{A Constructive Approach to a Class of Overdetermined Problems}}}}
\date{}
\author{
Alessandro Fortunati\thanks{DICCA -- Dipartimento di Ingegneria Civile, Chimica e Ambientale, Università degli Studi di Genova, Via Montallegro 1, 16145 Genova, Italy (see Acknowledgements).
}
\and
Filomena Pacella\thanks{Dipartimento di Matematica Guido Castelnuovo, Sapienza Università di Roma, Piazzale Aldo Moro 5, 00185 Roma, Italy. Email: filomena.pacella@uniroma1.it}
}

\usepackage{amsmath}
\usepackage{graphicx}
\begin{document}

\maketitle

\begin{abstract}
In this paper we study an overdetermined problem which is directly related to the well known torsion problem studied by J. Serrin. A perturbed version of the latter is tackled by using asymptotic series as well as tools borrowed from the celebrated Nekhoroshev Theorem. In a similar fashion to this class of results, we establish the existence of infinitely many approximants for the perturbed problem's solution, whose approximation error is so small which can be regarded as negligible for practical applications. The approach is fully constructive and this feature is demonstrated via an example in the final section.      
\end{abstract}
\bigskip
\begin{center}
\begin{minipage}{\dimexpr\textwidth-2cm\relax}

\end{minipage}
\end{center}
\bigskip

\section{Introduction}
In the last decades, overdetermined problems have given rise to many interesting questions in Mathematics and its applications. The first result is the famous symmetry result of J. Serrin \cite{serrin} about the following problem
\beq{eq:unperturbed}
\left\{
\begin{array}{rcl}
-\Delta u & = & 1 \qquad \mbox{on} \quad \Omega\\
u &=& 0 \qquad \mbox{on} \quad  \partial \Omega\\
\partial_{\nu}  u &=& C \qquad \mbox{on} \quad \partial \Omega\\  
\end{array}
\right. \mx{,}
\eeq
where $\Omega$ is a bounded smooth domain in $\mathbb{R}^N$, $\nu$ is the inner normal to $\partial \Omega$ and $C$ is a positive constant. \\
Problem (\ref{eq:unperturbed}) is often called \emph{torsion problem} because of its physical motivations, see \cite{serrin}.\\
The classical Serrin's result claims that the only smooth domain $\Omega$ admitting a solution of (\ref{eq:unperturbed}) is a ball with radius $R>0$ depending solely on the constant $C$, which prescribes the value of the normal derivative.       \\
This result has been extended in various directions: more general elliptic operators, nonlinear problems, quantitative versions, relative settings, non-smooth domains. etc. and the corresponding literature is extremely large. As for the mentioned works, we refer to the recent papers \cite{figalli}, \cite{pac} and references therein, as well as to the surveys \cite{fifteen} and \cite{sixteen}.\\
In this paper we consider a ``perturbed torsion problem'' where the Laplacian operator is perturbed by adding a small potential and the domain is a perturbed ball in $\mathbb{R}^2$. More precisely, we study the problem:
\beq{eq:perturbedserrin}
\left\{
\begin{array}{rclll}
-\Delta u+\mu w u &=&1&\text{on}&\Omega_{\mu}\\
u&=&0&\text{on}&\partial\Omega_{\mu}\\
\partial_{\nu}u&=&C&\text{on}&\partial\Omega_{\mu}
\end{array}
\right. \mx{,}
\eeq
where $u=u(r,\theta)$, $w=w(r,\theta)$ with $(r,\theta)$ is the standard set of polar coordinates of $\mathbb{R}^2$, and 
\[
\Omega_{\mu}:=\{(r,\theta): 0 \leq r \leq 1+\mu g(\theta): \, \theta \in \mathbb{T} \}  \mx{,}
\] 
being $\mu$ a ``small'' fixed positive parameter (in other terms $0 < \mu \ll 1$), $\mathbb{T} \equiv [0,2\pi]$ and $g:\mathbb{T} \mapsto \mathbb{R}$. The regularity of the above mentioned functions will be discussed later on.\\
Our aim is to show that for sufficiently small $\mu$, we can find domains $\Omega_{\mu}$ (which are perturbed balls) and infinitely many potentials $w^*$ for which we can produce approximated solutions of (\ref{eq:perturbedserrin}).\\
To do this, we use a constructive approach based on truncated expansions in power series which allow to get precise estimates of their remainders. \\
Our result can be seen as a quantitative version of Serrin's problem in the presence of a perturbed Laplace operator. Note that more general elliptic operators have been considered in \cite{serrin} and other papers, but they do not cover our  case.\\
We postpone to Sec. \ref{sec:two} the precise setting of the problem and the statement of the main result. 
From a more technical viewpoint, the paper relies on the possibility to construct approximated solutions by means of asymptotic series, which essentially go back to Poincaré. These, have been profitably used in the context of the well known Nekhoroshev Theorem \cite{Nek77}, \cite{Nek79}, concerning the stability of nearly-integrable Hamiltonian systems, with several remarkable applications, e.g., in the field of Celestial Mechanics, see also \cite{GalGio84}. \\
The approach has been shown to be successful in radically different problems involving perturbed PDEs, see e.g. \cite{ForBacMis24}. In the latter, the approximation error is ``so small'' with respect to the perturbation parameter, such that the constructed function turn out to be practically indistinguishable from the real solution (provided it exists). The strategy is classical but with peculiar technicalities: Firstly, the problem is formally reduced to a hierarchy of explicitly resolvable problems, afterwards, the approximation features of the constructed objects are studied in a suitable function space. Is some cases, see e.g. \cite{For18}, the employed scheme can be shown to be convergent to the actual solution of the problem at hand. A key property of this approach lies in its constructiveness, as the proof possesses the structure of an algorithm, which can be used to build those approximants explicitly. A fully detailed example will be presented in the final section of this work to demonstrate this fundamental property, together with a plot of the constructed solution. \\  
To the best of our knowledge, this is the first time a constructive approach of this type has been employed in the context of overdetermined problems.  

\section{Problem set-up and main result}\label{sec:two}
Let us consider the overdetermined problem (\ref{eq:perturbedserrin}) supposing  $C=1/2$ throughout the paper for simplicity and without loss of generality.\\
Fixed $\mu$, let us define $\mathcal{I}_{\mu}:=[0,1+\mu \max_{\theta \in \mathbb{T}}\snorm{g(\theta)}{}]$. The paper is concerned with the construction of approximants of potential solutions to (\ref{eq:perturbedserrin}) in the form of truncated power series 
\[
\begin{aligned}
u(r,\theta,\mu)& =u_0(r)+\mu u_1(r,\theta)+\mu^2 u_2(r,\theta)+\ldots + \mu^K u_K(r,\theta) \\
w(r,\theta,\mu)& =w_1(r,\theta)+\mu w_2(r,\theta)+\ldots + \mu^{K-1} w_K(r,\theta)
\end{aligned}
\mx{,}
\]
where $K \in \mathbb{N} \setminus \{0\}$ to be determined, $u_j,w_j: D_{\mu} \mapsto \mathbb{R}$, with $D_{\mu}:=\mathcal{I}_{\mu} \times \mathbb{T} \mapsto \mathbb{R}$. We assume this structure for all $j \geq 1$ (and not only $1 \leq j \leq K$) as in some parts of the paper (e.g. Sec. \ref{sec:formal}), the above mentioned (finite) series will be more easily treated as infinite series (i.e. Taylor expansions).\\    
It is immediate to realise that the problem with $\mu=0$, is nothing but the classical Serrin's problem (\ref{eq:unperturbed}), whose solution is the pure radial function 
\beq{eq:radial}
u_0=(1-r^2)/4 \mx{,}
\eeq
if $C=1/2$. 
\begin{defn} Given $w(\theta,r)$ and $u(\theta,r)$, let us introduce the following functions
\begin{align}
\mathcal{R}(u,w) & :=-\Delta u(r,\theta) + \mu w(r,\theta) u(r,\theta) -1 \mx{,} \label{eq:firstrem}\\
\ml{E}_{D}(u) & :=u(r,\theta)|_{\pl \Omega_{\mu}} \mx{,} \label{eq:dbd}\\  
\ml{E}_{N}(u) & :=\left(\pl_{\nu}u(r,\theta)|_{\pl \Omega_{\mu}}-C \right) (1+\mu g) \sqrt{1+2 \mu g +\mu^2 \left(g^2+(g')^2\right)^2}  \label{eq:nbd} \mx{,}
\end{align}
which will be called ``Resolvability'', ``Dirichlet'' and ``Neumann'' defects, respectively. Clearly $u,w$ are solution of (\ref{eq:perturbedserrin}) iff the three defects vanish identically.  
\end{defn}
\medskip 
\noindent Let us now consider, for all $\rho,\sigma \in (0,1]$ the set $\ml{D}_{\rho,\sigma;\mu}:=\ml{G}_{\rho;\mu} \times \mathbb{T}_{\sigma}$, where
$$
\ml{G}_{\rho;\mu} := \bigcup_{r \in \mathcal{I}_{\mu}} \mathcal{B}_{\rho}(r), \quad 
\mathcal{B}_{\rho}(r):=\{\hat{r} \in \CC :\snorm{\hat{r}-r}{} \leq \rho\}, \quad \mathbb{T}_{\sigma} := \{\hat{\theta} \in \CC: |\Im \hat{\theta}| \leq  \sigma\} 
\mx{.}
$$
Here $\Im z$ denotes the imaginary part of the complex number $z$. 
This construction can be naturally interpreted as the complexification of $D_{\mu}$. In fact, it is immediate to realise that $D_{\mu} \equiv \lim_{(\rho,\sigma) \rw (0,0)} \ml{D}_{\rho,\sigma;\mu}$. \\
Let us now denote with $\mathfrak{H}_{\rho,\sigma;\mu}$ the space of continuous functions on $\ml{D}_{\rho,\sigma;\mu}$, holomorphic in its interior and real on the restriction $D_{\mu}$.
Looking for solutions in this class of functions allows us to deal with objects which can be treated as real functions of real variables (see Sec. \ref{sec:formal} or Sec. \ref{sec:example}) but possessing extra-properties (crucial for Sec. \ref{sec:quantitative}) inherited from the requested regularity in the above described complex extension.\\ 
Let $h \in \mathfrak{H}_{\rho,\sigma;\mu}$, we can expand $h=\hat{a}_0(r)+\sum_{n=1}^{+\infty}[\hat{a}_k(r) \cos (n \theta) + \hat{b}_k(r) \sin (n \theta)]$ and define the \emph{Fourier norm} 
\beq{eq:fouriernorm}
\norm{h}{\rho,\sigma}:=
\snorm{\hat{a}_0(r)}{\rho}+\sum_{n=1}^{+\infty}(\snorm{\hat{a}_n(r)}{\rho} + |\hat{b}_n(r)|_{\rho})e^{n \sigma}\mx{.}
\eeq
where $\snorm{f(r)}{\rho}:=\sup_{r \in \ml{G}_{\rho}}\snorm{f(r)}{}$.
\begin{rem}\label{rem:one}
Clearly, if $f=f(\theta)$ then $\snorm{f}{\rho} \equiv \snorm{f}{}$.
\end{rem}
\begin{rem}\label{rem:two}
As a standard result in the context of holomorphic functions, it is possible to show that if $h \in \mathfrak{H}_{\rho,\sigma;\mu}$, then $\norm{h}{\rho,\sigma'} < +\infty$, for all $\sigma' <\sigma$. Conversely, if $\norm{h}{\rho,\sigma} \leq \ml{K} \in \mathbb{R}$, then $|\hat{a}_n(r)|_{\rho}, |\hat{b}_n(r)|_{\rho} \leq \ml{K} \exp(-\sigma n)$. In particular, the series (\ref{eq:fouriernorm}) converges in the interior of $\ml{D}_{\rho,\sigma;\mu}$, hence $h \in \mathfrak{H}_{\rho,\sigma;\mu}$.    
\end{rem}
\begin{satz}\label{thm:main}%
Suppose $g \in \mathfrak{H}_{\rho,\sigma;\mu}$ and satisfying 
\beq{eq:g}
\norm{g(\theta)}{(\rho,\sigma)}, \norm{g'(\theta)}{(\rho,\sigma)} \leq G \equiv 1/4 \mx{.}
\eeq
Then there exists $\mu_0>0$ and infinitely many $w^* \in \mathfrak{H}_{\rho/2,\sigma/2;\mu}$ such that, for each $w^*$, one can determine a unique $u^* \in \mathfrak{H}_{\rho/2,\sigma/2;\mu}$ such that the pair $(u^*,w^*)$ satisfies the following bounds
\beq{eq:finalbound}
\mathscr{C}_1^{-1} \norm{\mathcal{R}(u^*,w^*)}{\left(\frac{\rho}{2},\frac{\sigma}{2}\right)}, 
\mathscr{C}_4^{-1} \norm{\mathcal{E}_D(u^*)}{\left(\frac{\rho}{2},\frac{\sigma}{2}\right)},
\mathscr{C}_5^{-1} \norm{\mathcal{E}_N(u^*)}{\left(\frac{\rho}{2},\frac{\sigma}{2}\right)}
\leq  \mu^{\mathscr{C}_2(\log(1/\mu)-2\log \mathscr{C}_3)}   \mx{,}
\eeq
for all $\mu \in (0,\mu_0]$ and some $O(1)$ real constants $\mathscr{C}_j$, i.e. independent of $\mu$, for $j=1,2,\ldots, 5$. In other terms, $(u^*,w^*)$ satisfy the problem (\ref{eq:perturbedserrin}) up to an error which is super-polynomially small with respect to $\mu$, i.e. is $o(\mu^q)$ as $\mu \rw 0$, no matter how large $q$ is chosen.
\end{satz}
\begin{rem} To give an idea of the smallness of  (\ref{eq:finalbound}), let us suppose for simplicity $\mathcal{C}_{2,3}=1$, so that the r.h.s. of (\ref{eq:finalbound}) simplifies to $\exp(-(\log(1/\mu))^{2})$. The table below shows how dramatically the accuracy of the approximated solution increases as $\mu$ shrinks:
\[
\begin{array}{c|c}
\mu & \exp(-(\log(1/\mu))^{2}) \\[4pt]\hline
e^{-2}\;=\;1.4\times 10^{-1} & e^{-4}\approx 1.8\times 10^{-2} \\
e^{-4}\;=\;1.8\times 10^{-2} & e^{-16}\approx 1.1\times 10^{-7} \\
e^{-8}\;=\;3.4\times 10^{-4} & e^{-64}\approx 1.6\times 10^{-28} \\
e^{-16}\;=\;1.1\times 10^{-7} & e^{-256}\approx 3.8\times 10^{-112}
\end{array}
\]
As anticipated, although the constructed solution is not exact, it is practically indistinguishable from the exact one (provided it exists), as $\mu$ approaches zero. 
\end{rem}
\begin{rem}
The explicit values of $\mu_0$ and of the constants $\mathscr{C}_{1,2,3,4,5}>0$ are given in (\ref{eq:muzeronew}), (\ref{eq:mathscrcd}), (\ref{eq:mathscrc}), (\ref{eq:cfour}) and (\ref{eq:cfive}), respectively. As far as the previous remark is concerned, $\mathscr{C}_2 \ll 1$ and $\mathscr{C}_3 \gg 1$, but this is perturbatively irrelevant as both constants are independent of $\mu$ and hence $O(1)$.
\end{rem}
\begin{rem}\label{rem:sigma}
Another comment is in order about the choice of the analyticity radii $\rho,\sigma$. As the only analyticity assumption involves $g(\theta)$, which is non-radial, one might be led to believe that $\rho$ might me increased arbitrarily, with the hope to obtain better estimates. However, it is important to recall Liouville's Theorem, for which any entire bounded function must be constant. Hence, a bound on $\rho$ is immediate from the very expression of $u_0$, whose boundedness (we shall suppose $\norm{u_0}{(\rho,\sigma)}  \leq 1$, see also Rem \ref{rem:rho}) implies a bound on $\rho$. Although a refinement of the quantitative bounds might be possible, they would be still conceptually equivalent to the paradigmatic case $\rho,\sigma \leq 1$ considered here.   
\end{rem}
The rest of the paper is devoted to the proof of Theorem \ref{thm:main}. 
As it is very common in perturbative arguments, the proof will be divided into two parts. In the first one, also called ``formal'', a constructive algorithm is introduced to solve the problem up to some arbitrarily high power of $\mu$. However, despite the remainder is, say, $C \mu^{k}$ with arbitrarily large $k$, it is shown that there exists some critical $k^*$ such that the constant $C$ starts to grow so fast that increasing $k$ is no longer convenient. This is a typical phenomenon occurring, for instance, in Nekhoroshev-type arguments, see e.g. \cite{GalGio84} and goes back to the asymptotic series property already pointed out by Poincaré. The solution to this problem is to compute a specific $k$, we shall call $K_{\text{opt}}$ and set it as a final optimal normalisation order. This requires a bound on the remainder and it will be carried out in a subsequent ``quantitative'' section.       

\section{Formal algorithm}\label{sec:formal}
\subsection{Hierarchy of equations and formal lemma}
Let us consider the equation appearing in (\ref{eq:perturbedserrin}) and proceed formally by introducing the following (infinite) Taylor expansions
\begin{align}
u(r, \theta) &= \sum_{k \geq 0} \mu^k u_k(r, \theta), \label{eq:expandu} \\
w(r, \theta) &= \sum_{k \geq 1} \mu^{k-1} w_k(r, \theta). \label{eq:expandw}
\end{align}
By substituting in the mentioned equation, via a balance of the powers of $\mu$, one gets the following hierarchy of equations 
\begin{align}
\Delta u_1(r, \theta) - w_1(r, \theta)  u_0(r) & = 0 \label{eq:hierone}\\
\Delta u_2(r, \theta) - w_2(r, \theta)  u_0(r) & = u_1(r, \theta) w_1(r, \theta), \label{eq:hiertwo} \\
\ldots \nonumber
\end{align}
i.e., 
\beq{eq:hierarchy}
\Delta u_k(r, \theta) - w_k(r, \theta)  u_0(r) = \mathcal{F}^{(k)}(r,\theta)\mx{,}
\eeq
for all $k$, where
\beq{eq:source}
\mathcal{F}^{(k)}(r,\theta):=\sum_{j=1}^{k-1} w_j(r, \theta) u_{k-j}(r, \theta) \mx{,}
\eeq
is a known function, as it depends on  $u_0,\ldots, u_{k-1}$ and $w_1,\ldots, w_{k-1}$. Note that all of them are functions constructed during the previous stages of the hierarchy. \\
In this setting, we can state the following 
\begin{lem}[Formal]\label{lem:formal} 
For any $k \geq 1$, there exist a family of infinitely many $w_k(r, \theta)$ such that, for any of them, it is possible to determine a unique solution $u_k(r, \theta)$ of the equation (\ref{eq:hierarchy}) in such a way, by defining the truncated expansions 
$u^{[k]}(r, \theta):=\sum_{j = 0}^k \mu^j u_j(r, \theta)$ and $w^{[k]}(r, \theta):=\sum_{j = 1}^k \mu^{j-1} w_j(r, \theta)$, understood $u_j \equiv w_j \equiv 0$ for all $j \geq k+1$, the pair $(u^{[k]},w^{[k]})$ give rise to the following expressions of the defects
\begin{align}
\mathcal{R}^{[k]}:=\ml{R}(u^{[k]},w^{[k]})& = \sum_{n = k+1}^{2k} \mu^n \sum_{j = n - k}^{k} w_{n-j}(r, \theta) u_j(r, \theta) \label{eq:remainder},\\
\ml{E}_{D}^{[k]} :=\ml{E}_{D}(u^{[k]})&=\sum_{j=k+1}^{\infty}\mu^{j}\sum_{m=j-k}^{j}
\frac{g^{m}(\theta)}{m!}\partial_r^{m}u_{j-m}(1,\theta), \label{eq:dirichletdefect}\\
\ml{E}_{N}^{[k]}
:=
\ml{E}_{N}(u^{[k]})
&=
\sum_{j=k+1}^{\infty}\mu^j\Bigg\{
g'\sum_{h=0}^{k}\frac{g^{j-1-h}}{(j-1-h)!}
\partial_r^{j-1-h}\partial_\theta u_h(1,\theta) \nonumber \\
& -\sum_{h=0}^{k}\frac{g^{j-h}}{(j-h)!}
\partial_r^{j-h+1}u_h(1,\theta) -2g\sum_{h=0}^{k}\frac{g^{j-1-h}}{(j-1-h)!}
\partial_r^{j-h}u_h(1,\theta)\nonumber\\
&-g^2\sum_{h=0}^{\min\{j-2,k\}}\frac{g^{j-2-h}}{(j-2-h)!} 
\partial_r^{j-h-1}u_h(1,\theta) -C C_j(\theta)
\Bigg\},
\label{eq:neumanndefect} 
\end{align}
where $C_0(\theta) \equiv 1$ and 
\begin{align}
C_j(\theta)
& :=
\sum_{m=\lceil j/2\rceil}^{j}
\binom{\tfrac12}{m}\binom{m}{j-m}
\bigl(2g\bigr)^{2m-j}
\bigl(g^2+(g')^2\bigr)^{j-m} \nonumber \\
& +g\sum_{m=\lceil (j-1)/2\rceil}^{j-1}
\binom{\tfrac12}{m}\binom{m}{j-1-m}
\bigl(2g\bigr)^{2m-(j-1)}
\bigl(g^2+(g')^2\bigr)^{j-1-m} \label{eq:c},
\end{align}
for all $j \geq 1$ and $\binom{1/2}{m}:=( 2 m!)^{-1} (1/2-1)(1/2-2)\cdots(1/2-(m-1))
$ with $\binom{1/2}{0}=1$.\\
In particular  $\mathcal{R}^{[k]}, \ml{E}_{D}^{[k]}, \ml{E}_{N}^{[k]}=O \left(\mu^{k+1}\right)$ i.e. the pair $(u^{[k]},w^{[k]})$ satisfies (\ref{eq:perturbedserrin}) up to $O(\mu^k)$.
\end{lem}
\begin{ex} We have $C_1(\theta)=2 g$ and $C_2(\theta)=g^2(\theta)+(1/2)(g')^2$.    
\end{ex}
\proof The proof will consist in deriving, at each stage of the hierarchy, corresponding Dirichlet and Neumann conditions and imposing them in the equation written in the Fourier space. A pseudocode version of this proof in form of algorithm shall be given in the Appendix. Let us firstly expand
\begin{align}
u_k(r, \theta) &= A_0^{(k)}(r)+\sum_{n=1}^{\infty} \left[ A_n^{(k)}(r) \cos(n\theta) + B_n^{(k)}(r) \sin(n\theta) \right], \label{eq:ukfourier} \\
w_k(r,\theta) &= a_0^{(k)}(r)+\sum_{n=1}^{\infty} \left[ a_n^{(k)}(r) \cos(n\theta) + b_n^{(k)}(r) \sin(n\theta) \right], \label{eq:wkfourier} \\
\mathcal{F}^{(k)}(r,\theta) &=  \xi_0^{(k)}(r)+ \sum_{n=1}^{\infty} \left[ \xi_n^{(k)}(r) \cos(n\theta) + \eta_n^{(k)}(r) \sin(n\theta) \right]. \label{eq:fkfourier}
\end{align}
By recalling $\Delta u_k = \partial_r^2 u_k + r^{-1} \partial_r u_k + r^{-2} \partial_\theta^2 u_k$ and substituting the expansions above in (\ref{eq:hierarchy}) we obtain
\begin{align}
A_0^{(k)\prime\prime}(r) + r^{-1} A_0^{(k)\prime}(r) & = 
u_0(r)  a_0^{(k)}(r) + \xi_0^{(k)}(r)=: \mathfrak{f}_0^{(k)} (r), \label{eq:ceazero}\\
A_n^{(k)\prime\prime}(r) + r^{-1} A_n^{(k)\prime}(r) - r^{-2} n^2 A_n^{(k)}(r) &= u_0(r)  a_n^{(k)}(r) + \xi_n^{(k)}(r)=:\mathfrak{f}_n^{(k)} (r), \label{eq:cean} \\
B_n^{(k)\prime\prime}(r) + r^{-1} B_n^{(k)\prime}(r) - r^{-2} n^2 B_n^{(k)}(r) &= u_0(r)  b_n^{(k)}(r) + \eta_n^{(k)}(r) =:\mathfrak{g}_n^{(k)} (r), \label{eq:cebn}
\end{align}
which are clearly in the Cauchy-Euler form. Solutions which stay bounded in the limit $r \rw 0$ are realised via a suitable choice of the free constants. As it easy to check, those solutions have the following expressions
\begin{align}
A_0^{(k)}(r) & = C_0^{(k)} + \int_0^r \frac{1}{s} \int_0^s s' \mathfrak{f}_0^{(k)}(s') ds' ds \label{eq:azero}\\
A_n^{(k)}(r) &= r^n \left( C_n^{(k)} - \frac{1}{2n} \int_r^1 \mathfrak{f}_n^{(k)}(s) s^{1 - n} ds \right)
- \frac{r^{-n}}{2n} \int_0^r \mathfrak{f}_n^{(k)}(s) s^{1 + n} ds, \label{eq:an} \\
B_n^{(k)}(r) &= r^n \left( D_n^{(k)} - \frac{1}{2n} \int_r^1 \mathfrak{g}_n^{(k)}(s) s^{1 - n} ds \right)
- \frac{r^{-n}}{2n} \int_0^r \mathfrak{g}_n^{(k)}(s) s^{1 + n} ds, \label{eq:bn}
\end{align} 
where $C_n^{(k)}$, $D_n^{(k)}$ are free (sequences of) constants to be determined.  
\begin{rem}\label{rem:lack}
It is clear from these expressions that the boundedness constraint as $r \rw 0$ ``removes'' one parameter from the solutions of the Cauchy-Euler equations. On the other hand, as it is natural to expect, either the Dirichlet or the Neumann boundary conditions would require one free parameter each to be satisfied. Basically, this is the reason why the problem does not possess solution if $w$ is fixed a priori and this freedom is represented by the possibility to choose $a_n^{(k)}$ and $b_n^{(k)}$. As we shall see, those will allow to compensate the ``lack'' of available constants in the Cauchy-Euler solutions. 
\end{rem}
\subsection{Dirichlet boundary condition}
By definition of $\Omega_{\mu}$, the boundary is parameterised by $r = 1 + \mu g $. Firstly let us recall the expansion for $u$, 
\[
u(1 + \mu g, \theta) = \sum_{j=0}^\infty \mu^j u_j(1 + \mu g, \theta) \mx{,}
\]
On the other hand, each $u_j$ can be expanded in Taylor series
\[
u_j(1 + \mu g, \theta) = u_j(1, \theta) + \mu g \partial_r u_j(1, \theta) + \frac{\mu^2 g^2}{2} \partial_r^2 u_j(1, \theta) + \cdots
\]
Hence, the condition is equivalent to
\begin{align}
0=\mathcal{E}_D(u)=u(1 + \mu g, \theta) & = \sum_{m=0}^\infty \mu^m \left[ u_m(1, \theta) + \sum_{h=1}^{\infty} \frac{\mu^h g^h}{h!} \partial_r^h u_m(1, \theta) \right] \nonumber \\
&=\sum_{m=0}^\infty \mu^m \sum_{h=0}^{\infty} \frac{\mu^h g^h}{h!} \partial_r^{h} u_m(1, \theta) \nonumber \\
&=\sum_{m=0}^\infty\sum_{h=0}^{\infty} \mu^{m+h} \frac{g^h}{h!} \partial_r^{h} u_m(1, \theta)
\nonumber \\
&=\sum_{j=0}^{\infty}\mu^j
\sum_{h=0}^{j}\frac{g^h}{h!} \partial_r^{h}u_{j-h}(1,\theta) \mx{.}
\label{eq:lastexpdir}
\end{align}
As a consequence, by isolating the term $h=0$ in the last sum, the Dirichlet boundary condition reads as a condition on $u_j(1,\theta)$, i.e.
\beq{eq:dirichlet}
u_j(1, \theta) = - \sum_{h=1}^{j}\frac{g^h}{h!} \partial_r^{h}u_{j-h}(1,\theta) =: \Phi^{(j)}(\theta) \mx{,}
\eeq
whic is a function of previously determined objects of the hierarchy. Hence, at each stage $j$, $u_j(1, \theta)$ is a known function of $\theta$, hence its Fourier coefficients can be completely determined 
\beq{eq:expansionfourierdirichlet}
u_j(1, \theta)=:\widetilde{A}_0^{(j)}+\sum_{n=1}^{\infty} \left[ \widetilde{A}_n^{(j)} \cos(n\theta) + \widetilde{B}_n^{(j)} \sin(n\theta) \right] \mx{.}
\eeq 
\subsection{Neumann boundary conditions}
It is easy to check that for our domain $\Omega_{\mu}$ we have
\[
\pl_{\nu} u|_{\pl \Omega_{\mu}}
=
\frac{-(1+\mu g)u_r(1+\mu g, \theta)
+\mu (1+\mu g)^{-1} g'u_\theta(1+\mu g)}
{\sqrt{(1+\mu g)^2+\mu^2(g')^2}}\mx{.}
\]
Hence, the Neumann condition reads as 
\beq{eq:neumannconditionfirst}
\begin{aligned}
0=\ml{E}_N:=&
\mu g' \partial_\theta u(1 + \mu g, \theta)
- (1+\mu g)^2 \partial_r u(1 + \mu g, \theta)\\
&- C (1+\mu g) \sqrt{1+2\mu g+\mu^{2}\bigl(g^{2}+(g')^{2}\bigr)} \mx{.}
\end{aligned}
\eeq
We proceed by expanding each term near $r = 1$ and using the power expansions for $u$. This yields 
\begin{align}
\partial_\theta u(1 + \mu g, \theta)
& = \sum_{j=0}^\infty \mu^j \sum_{h=0}^{j} \frac{g^{j - h}}{(j - h)!} \cdot \partial_r^{j - h} \partial_\theta u_h(1, \theta) \label{eq:neuone}\mx{,}\\
(1+\mu g)^2\partial_r u(1+\mu g,\theta)
&=\sum_{j=0}^{\infty}\mu^j\Bigg[
\sum_{h=0}^{j}\frac{g^{j-h}}{(j-h)!}\partial_r^{j-h+1}u_h(1,\theta)
\nonumber\\
&\qquad\qquad
+2g\sum_{h=0}^{j-1}\frac{g^{j-1-h}}{(j-1-h)!}\partial_r^{j-h}u_h(1,\theta)
\nonumber\\
&\qquad\qquad
+g^2\sum_{h=0}^{j-2}\frac{g^{j-2-h}}{(j-2-h)!} \partial_r^{j-h-1}u_h(1,\theta)
\Bigg], \label{eq:neutwo}\\
(1+\mu g)\sqrt{1+2\mu g+\mu^{2}\bigl(g^{2}+(g')^{2}\bigr)}
&=\sum_{j=0}^{\infty}\mu^{j} C_j(\theta)\label{eq:neuthree} \mx{,}
\end{align} 
where $C_j(\theta)$ has been defined in (\ref{eq:c}). We recall $u_0(r,\theta) \equiv u_0(\theta)$ so that $\pl_{\theta} u_0(r,\theta) \equiv 0$. 
\begin{ex}\label{rem:firstorders}%
Let us restore $C=1/2$, then the first orders of (\ref{eq:neumannconditionfirst}) read as
\begin{align*}
j=0: \, \,  \partial_r u_0(1,\theta) &=-1/2,\\
j=1: \, \, \partial_r u_1(1,\theta)
&= g'\partial_\theta u_0(1,\theta)
- g\partial_r^{2}u_0(1,\theta)
-2g\partial_r u_0(1,\theta)
-2C g\\
&=(1/2)g,\\
j=2: \, \, \partial_r u_2(1,\theta)
&=g'\Bigl[g\partial_r\partial_\theta u_0(1,\theta)
+\partial_\theta u_1(1,\theta)\Bigr]
-(1/2)g^2\partial_r^{3}u_0(1,\theta)\\
&\qquad-g\partial_r^{2}u_1(1,\theta)
-2g^2\partial_r^{2}u_0(1,\theta)
-2g\partial_r u_1(1,\theta)\\
&\qquad-g^2\partial_r u_0(1,\theta)
-C\Bigl(g^2+(1/2)(g')^2\Bigr)\\
&=g'\partial_\theta u_1(1,\theta)
-g\partial_r^{2}u_1(1,\theta)
-2g\partial_r u_1(1,\theta)
+g^2-(1/4)(g')^2,
\end{align*}
where the expression above have been simplified by using (\ref{eq:radial}). 
\end{ex}
Similarly to the effect of Dirichlet condition on $u_j(1,\theta)$ found in (\ref{eq:dirichlet}), the example above shows how the Neumann condition can be interpreted as a constraint on $\pl_r u_j(1,\theta)$ for $j=0,1,2$. This is actually a general property. To prove this, let us now substitute the expansions \eqref{eq:neuone}, \eqref{eq:neutwo} and \eqref{eq:neuthree} into \eqref{eq:neumannconditionfirst}, then collect the powers of $\mu$
\beq{eq:expansionneumann}
\begin{aligned}
0&=\sum_{j=0}^{\infty}\mu^j\left\{
g'\sum_{h=0}^{j-1}\frac{g^{j-1-h}}{(j-1-h)!}
\partial_r^{j-1-h}\partial_\theta u_h(1,\theta)-\sum_{h=0}^{j}\frac{g^{j-h}}{(j-h)!}
\partial_r^{j-h+1}u_h(1,\theta) \right. \\
&\left. -2g\sum_{h=0}^{j-1}\frac{g^{j-1-h}}{(j-1-h)!}
\partial_r^{j-h}u_h(1,\theta) 
-g^2\sum_{h=0}^{j-2}\frac{g^{j-2-h}}{(j-2-h)!}
\partial_r^{j-h-1}u_h(1,\theta) -C C_j(\theta)
\right\} \mx{,}
\end{aligned}
\eeq
in which it is understood $\sum_{j=p}^{q} \cdot =0$ for all $q < p$. \\
Let us now isolate the term $h=j$ in the second sum, i.e. $-\partial_r u_j(1,\theta)$. Hence, the quantity between curly brackets can be written as 
\[
0=\sum_{j=0}^{\infty}\mu^j\left\{-\partial_r u_j(1,\theta) + \Psi^{(j)} \right\} \mx{,}
\]
which holds provided that the (infinite) family of equations 
\beq{eq:neumanneq}
\partial_r u_j(1, \theta)=\Psi^{(j)}(\theta) \mx{,}
\eeq
are satisfied, where
\begin{align}
\Psi^{(j)}(\theta)
:={}& g'\sum_{h=0}^{j-1}\frac{g^{j-1-h}}{(j-1-h)!}
\partial_r^{j-1-h}\partial_\theta u_h(1,\theta)\nonumber\\
&-\sum_{h=0}^{j-1}\frac{g^{j-h}}{(j-h)!}
\partial_r^{j-h+1}u_h(1,\theta)
-2\sum_{h=0}^{j-1}\frac{g^{j-h}}{(j-1-h)!}
\partial_r^{j-h}u_h(1,\theta)\nonumber\\
&-\sum_{h=0}^{j-2}\frac{g^{j-h}}{(j-2-h)!}
\partial_r^{j-h-1}u_h(1,\theta) -C C_j(\theta). \label{eq:psi}
\end{align}
As a consequence, the Neumann condition gives rise to a corresponding sequence of Fourier coefficients representing it
\beq{eq:expansionfourierdirichlet}
\pl_r u_j(1, \theta)=:\widetilde{A}_0^{(j)'}+\sum_{n=1}^{\infty} \left[ \widetilde{A}_n^{(j)'} \cos(n\theta) + \widetilde{B}_n^{(j)'} \sin(n\theta) \right] \mx{.}
\eeq 
\subsection{Construction of the (infinitely many) pair(s) $\left(u^{[k]},w^{[k]}\right)$.}
Let us recall that the constants
\beq{eq:parameters}
\widetilde{A}_0^{(j)}, \quad \widetilde{A}_0^{(j)\prime}, \quad \widetilde{A}_n^{(j)}, \quad \widetilde{A}_n^{(j)\prime}, 
\quad \widetilde{B}_n^{(j)}, \quad \widetilde{B}_n^{(j)\prime} \mx{,}
\eeq
have been determined in the latest sections right from the boundary conditions.\\
In this section, we are going to show how to exploit the variability of $a_n^{(j)}$ and $b_n^{(j)}$, as anticipated in Rem. \ref{rem:lack} to determine an infinite family of formal solutions. The approach will consists in showing that, given a paradigmatic class of $w_j$, for all $k \leq K$, it is possible to isolate a subset of functions in this class and a corresponding solution $u_k$, such that the $k-$th equation of the hierarchy is satisfied.\\  
For this purpose we shall consider the following expansions: set an arbitrary $M \in \mathbb{N}$ for all $k$ and define 
\beq{eq:anseries}
a_n^{(k)}(r)=\sum_{m=0}^{M} a_{n,m}^{(k)} r^m, \qquad b_n^{(k)}(r)=\sum_{m=0}^{M} b_{n,m}^{(k)} r^m.
\eeq
\begin{rem} 
We stress that $M$ is chosen to be independent on $k$ for simplicity of discussion only and without any loss of generality. It is worth observing that a fully acceptable choice could be $M=0$, implying in this way that $a_n^{(k)}(r) \equiv a_{n,0}^{(k)}$ and $b_n^{(k)}(r)  \equiv b_{n,0}^{(k)}$, i.e. constant. It is easy to realise that, the particular choice $M=0$ would imply $w_k=w_k(\theta)$, i.e. to be radially independent.   
\end{rem}
\noindent Let us now proceed by imposing the boundary conditions.
\subsubsection*{Zeroth mode (average):}
Recall (\ref{eq:azero}). By computing the first derivative and evaluating both at $r=1$, we impose (denoted with $\equiv$) that
\begin{align}
\widetilde{A}_0^{(k)} & \equiv {A}_0^{(k)}(1) = C_0^{(k)} + \int_0^1 \frac{1}{s} \int_0^s s' \mathfrak{f}_0^{(k)}(s') ds' ds, \label{eq:zerothzero}\\
\widetilde{A}_0^{(k)'} & \equiv {A}_0^{(k)'}(1) = \int_0^1 s \mathfrak{f}_0^{(k)}(s) ds. \label{eq:zerothone}
\end{align}
From (\ref{eq:zerothzero}), we immediately obtain the value of $C_0^{(k)}$. By substituting back into (\ref{eq:azero}), we get 
\beq{eq:azeropre}
A_0^{(k)}(r) = \widetilde{A}_0^{(k)} - \int_r^1 \frac{1}{s} \left( \int_0^s s' \mathfrak{f}_0^{(k)}(s')  ds' \right) ds \mx{.}
\eeq
This yields the final expression of $A_0^{(k)}(r)$, subject to the determination of $\mathfrak{f}_0^{(k)}(r)$, which is sill unknown because of $a_0^{(k)}(r)$. It is here that (\ref{eq:zerothone}) comes into play. 
First of all, by (\ref{eq:anseries}), we find that
\[
\int_0^1 s \mathfrak{f}_0^{(k)}(s) ds =  \frac{a_{0,0}^{(k)}(r)}{16} + 
\frac{1}{2} \sum_{m=1}^{M} \frac{a_{0,m}^{(k)}}{(2+m)(4+m)} + 
\int_0^1 s \xi_0^{(k)}(s) ds
\mx{.}
\]
On the other hand, the l.h.s. is equal to $\widetilde{A}_0^{(k)'}$ by (\ref{eq:zerothone}), hence, by solving w.r.t. $a_{0,0}^{(k)}(r)$ we obtain
\beq{eq:azz}
a_{0,0}^{(k)}(r)=16 \left[\widetilde{A}_0^{(k)\prime} - \int_0^1 s \xi_0^{(k)}(s) ds - \frac{1}{2} \sum_{m=1}^{M} \frac{a_{0,m}^{(k)}}{(2+m)(4+m)}  \right] \mx{.}
\eeq
Once the arbitrary sequence $\{a_{0,m}\}_{m=1,\ldots,M}$ has been set, $a_0^{(k)}(r)$ is found by (\ref{eq:anseries}) so that (\ref{eq:ceazero}) fully determines $\mathfrak{f}_0^{(k)}(r)$ and finally (\ref{eq:azeropre}) yields $A_0^{(k)}(r)$. 
\subsubsection*{Non-zero modes:}
By recalling (\ref{eq:an}) and (\ref{eq:bn}) and evaluating them at $r=1$ alongside their first derivatives, we require
\begin{align*}
\widetilde{A}_n^{(k)} & \equiv \lim_{r \to 1^-} {A}_n^{(k)}(r) = C_n^{(k)} - \frac{1}{2n} \int_0^1 \mathfrak{f}_n^{(k)}(s) s^{1+n} ds, \\
\widetilde{A}_n^{(k)'} & \equiv \lim_{r \to 1^-} {A}_n^{(k)'}(r)= n C_n^{(k)} + \frac{1}{2} \int_0^1 \mathfrak{f}_n^{(k)}(s) s^{1+n} ds, \\
\widetilde{B}_n^{(k)} & \equiv \lim_{r \to 1^-} {B}_n^{(k)}(r) = D_n^{(k)} - \frac{1}{2n} \int_0^1 \mathfrak{g}_n^{(k)}(s) s^{1+n} ds, \\
\widetilde{B}_n^{(k)'} & \equiv \lim_{r \to 1^-} {B}_n^{(k)'}(r) = n D_n^{(k)} + \frac{1}{2} \int_0^1 \mathfrak{g}_n^{(k)}(s) s^{1+n} ds.
\end{align*}
It is sufficient to multiply the first and the third ones by $n$ and sum them pairwise to get:
\beq{eq:cndn}
C_n^{(k)} = 2^{-1} \left( \widetilde{A}_n^{(k)} + n^{-1} \widetilde{A}_n^{(k)\prime} \right) \quad; \quad 
D_n^{(k)} = 2^{-1} \left( \widetilde{B}_n^{(k)} + n^{-1} \widetilde{B}_n^{(k)\prime} \right) \mx{.}
\eeq
It is now possible to substitute these into (\ref{eq:an}) and (\ref{eq:bn}), obtaining in this way:
\begin{align}
\int_0^1 \frac{1 - s^2}{4} a_n^{(k)}(s) s^{1+n} ds &= \widetilde{A}_n^{(k)\prime} - n \widetilde{A}_n^{(k)} - \int_0^1 \xi_n^{(k)}(s) s^{1+n} ds, \label{eq:foranm}\\
\int_0^1 \frac{1 - s^2}{4} b_n^{(k)}(s) s^{1+n} ds &= \widetilde{B}_n^{(k)\prime} - n \widetilde{B}_n^{(k)} - \int_0^1 \eta_n^{(k)}(s) s^{1+n} ds \mx{.} \label{eq:forbnm}
\end{align}
As for their l.h.s., one has, 
\begin{align*}
\int_0^1 \frac{1 - s^2}{4} a_n^{(k)}(s) s^{1+n} ds &= \frac{1}{4} \sum_{m=0}^M a_{n,m}^{(k)} \int_0^1 (1 - s^2) s^{1+n+m} ds \\
&= \frac{a_{n,0}^{(k)}}{\chi_n}+ \frac{1}{2} \sum_{m=1}^{M} \frac{a_{n,m}^{(k)}}{(2+n+m)(4+n+m)}  \mx{,}
\end{align*}
and similarly for $b_n^{(k)}$, here $\chi_n:=2 (2+n)(4+n)$. In conclusion, once the $2M$ free parameters $\{a_{\cdot,m}\}_{m=1,\ldots,M}$ and $\{b_{\cdot,m}\}_{m=1,\ldots,M}$ have been chosen, we obtain from (\ref{eq:foranm}) and (\ref{eq:forbnm}),
\begin{align}
a_{n,0}^{(k)} &= \chi_n \left[
\widetilde{A}_n^{(k)'} - n \widetilde{A}_n^{(k)}
- \int_0^1 \xi_n^{(k)}(s) s^{1+n} ds -\frac{1}{2} \sum_{m=1}^{M} \frac{a_{n,m}^{(k)}}{(2+n+m)(4+n+m)} 
\right] \mx{,} \label{eq:anm}\\
b_{n,0}^{(k)} &= \chi_n \left[
\widetilde{B}_n^{(k)'} - n \widetilde{B}_n^{(k)}
- \int_0^1 \eta_n^{(k)}(s) s^{1+n} ds -\frac{1}{2} \sum_{m=1}^{M} \frac{b_{n,m}^{(k)}}{(2+n+m)(4+n+m)} 
\right] \mx{.} \label{eq:bnm}
\end{align}
These are meant to complete the definitions of $a_n^{(k)}(r)$ and $b_n^{(k)}(r)$ via (\ref{eq:anseries}) and then those of $f_{n}^{(A)}(r)$ and $f_{n}^{(B)}(r)$ according to (\ref{eq:cean}) and (\ref{eq:cebn}), respectively, so that the determination of $A_n^{(k)}(r)$ and $B_n^{(k)}(r)$ is complete by (\ref{eq:an}) and (\ref{eq:bn}).
\endproof
\subsection{Remainders}\label{subsec:remainders}
\subsubsection{Resolvability defect $\mathcal{R}^{[k]}$}
Our aim is now to prove (\ref{eq:remainder}). To this end, let us substitute the expansions (\ref{eq:expandu}) and (\ref{eq:expandw}) into $\mathcal{R}$ defined in (\ref{eq:firstrem}), obtaining
\begin{align*}
\mathcal{R}=&-\Delta\Bigl(\sum_{j=0}^{\infty}\mu^j u_j\Bigr)+\mu\Bigl(\sum_{j=0}^{\infty}\mu^j u_j\Bigr)\Bigl(\sum_{m=1}^{\infty}\mu^{m-1}w_m\Bigr)-1\\
&=\sum_{j=0}^{\infty}\mu^j(-\Delta u_j)+\sum_{j=0}^{\infty}\sum_{m=1}^{\infty}\mu^{j+m}u_j w_m-1\\
&=\sum_{j=0}^{\infty}\mu^j(-\Delta u_j)+\sum_{j=1}^{\infty}\mu^{j}\sum_{m=1}^{j}u_{j-m}w_m-1\\
&=-\Delta u_0-1+\sum_{j=1}^{\infty} \left[ -\Delta u_j+\sum_{m=1}^{j}u_{j-m}w_m \right] \mx{,}
\end{align*}
where clearly $-\Delta u_0 - 1=0$ by (\ref{eq:radial}). It is immediate to see that the quantity between square parentheses of the latter is nothing but the equation of the hierarchy (\ref{eq:hierarchy}). Hence, imagining for a moment to be able to construct the whole sequence $u_j, w_j$ for all $j$, then $\mathcal{R}$ would vanish identically, i.e., as it is natural to expect, the full solution would have \emph{zero resolvability defect}.\\
To evaluate the defect of a truncated expansion, it is sufficient to split the sum above as follows
\[
\mathcal{R}=\sum_{j=1}^{k} \mu^{j} \left[ -\Delta u_j+\sum_{m=1}^{j}u_{j-m}w_m \right]+ \sum_{j=k+1}^{\infty} \mu^{j} \left[ -\Delta u_j+\sum_{m=1}^{j}u_{j-m}w_m \right] \mx{.}
\]
When evaluated on $u^{[k]}, w^{[k]}$, only the first sum of the latter vanishes, while the second sum constitutes the defect we wish to evaluate 
\[
\mathcal{R}^{[k]}= \left[
\sum_{j=k+1}^{\infty}  \mu^{j} \left( -\Delta u_j+\sum_{m=1}^{j}u_{j-m}w_m \right) 
\right]_{\substack{
u_{k+1}=u_{k+2}=\ldots=0\\
w_{k+1}=w_{k+2}=\ldots=0
}}
\mx{.}
\]
The suggested evaluation yields the desired expression. 
\begin{ex} The first expressions of $\mathcal{R}^{[k]}$ read as
\begin{align*}
\mathcal{R}^{[1]} & = \mu^2 w_1 u_1,\\
\mathcal{R}^{[2]} & = \mu^3 \left(w_1 u_2 + w_2 u_1 \right) + \mu^4 w_2 u_2,\\
\mathcal{R}^{[3]} &= \mu^4 \left(w_1 u_3 + w_2 u_2 + w_3 u_1\right)
+ \mu^5 \left(w_2 u_3 + w_3 u_2\right)
+ \mu^6 w_3 u_3.
\end{align*}
\end{ex}
\subsubsection{Dirichlet boundary defect $\ml{E}_{D}^{[k]}$}
Proceeding similarly to the resolvability defect, let us write (\ref{eq:lastexpdir}) by splitting the sum
\beq{eq:temporaryd}
\ml{E}_D \equiv u^{[k]}(1+\mu g,\theta)
=\sum_{j=0}^k \mu^j \sum_{h=0}^j \frac{g^h}{h!} \pl_r u_{j-h}(1,\theta)+
\sum_{j=k+1}^{\infty} \mu^j \sum_{h=0}^j \frac{g^h}{h!} \pl_r u_{j-h}(1,\theta) \mx{.}
\eeq
Once more, if we could be able to impose the infinite number of conditions (\ref{eq:dirichlet}), (\ref{eq:temporaryd}) would be identically zero. However, when using a truncated expression $u^{[k]}(1+\mu g, \theta)$, instead of the ``full'' $u(1+\mu g,\theta)$, the first $j \leq k$ conditions (\ref{eq:dirichlet}) are still imposed, so the first sum of (\ref{eq:temporaryd}) is still zero, however the remaining (\ref{eq:dirichlet}) for $j \geq k+1$ are lost and the second sum of (\ref{eq:temporaryd}) does not vanish anymore: this gives rise to the defect at hand.\\
More precisely, recalling its definition,  
\[
\ml{E}_{D}^{[k]} \equiv u^{[k]}(1+\mu g,\theta)
=\sum_{j=0}^{k}\mu^{j} \left[\Phi^{(j)}(\theta)- u_j(1,\theta)\right]
+\sum_{j=k+1}^{\infty}\mu^{j}\sum_{h=j-k}^{j}\frac{g^{h}(\theta)}{h!}\partial_r^{h}u_{j-h}(1,\theta),
\]
where the first sum is analogous to the first one in (\ref{eq:temporaryd}) and the second sum has been obtained from the second one of (\ref{eq:temporaryd}) by setting $u_j \equiv 0$ for all $j \geq k+1 $. As already pointed out, (\ref{eq:dirichlet}) implies that the first sum is identically zero, yielding in this way (\ref{eq:dirichletdefect}).

\subsubsection{Neumann boundary defect $\ml{E}_{N}^{[k]}$}
We are now dealing with the proof of (\ref{eq:neumanndefect}). Despite slightly cumbersome, the procedure is analogous the one we have used to obtain  (\ref{eq:dirichletdefect}). For this purpose let us now focus on \eqref{eq:expansionneumann} and let us split the sum in two, once more in the form $\sum_{j=0}^{k}\mu^{j} \{\cdot\}+\sum_{j=k+1}^{\infty}\mu^{j} \{\cdot\}$, where $\{\cdot\}$ denotes the curly brackets in \eqref{eq:expansionneumann}. As we are using $u^{[k]}$, only the conditions for $0 \leq j \leq k$ are satisfied, so that we will only have $\sum_{j=0}^{k}\mu^{j} \{\cdot\}=0$. As a consequence, $\ml{E}_{N}^{[k]} \equiv \sum_{j=k+1}^{\infty}\mu^{j} \{\cdot\}$. Also in this case, it is sufficient to use the property $u_j \equiv 0$ for all $j \geq k+1$ to get (\ref{eq:neumanndefect}).

\section{Quantitative estimates}\label{sec:quantitative}
This section is concerned with estimate (\ref{eq:finalbound}). In order to prove the latter, we need to find a bound on the objects involved in the scheme and their variation in size as the normalisation order grows. \\
The basic tool of our analysis is Cauchy's dimensional bound, which will be recalled below and broken down into three sub-cases for convenience.  
\begin{prop}[Cauchy bounds]
Suppose $u \in \mathfrak{H}_{\rho',\sigma';\mu}$, then the following bounds hold
\begin{align}
\left\| \partial_r^p \partial_\theta^q u \right\|_{(\rho'', \sigma'')} 
& \leq 
\frac{p!}{(\rho' - \rho'')^p} 
\frac{q!}{(\sigma' - \sigma'')^q} 
\left\| u \right\|_{(\rho', \sigma')} 
\quad \text{if } \rho'' < \rho', \; \sigma'' < \sigma', \label{eq:cauchyone}\\ 
\left\| \partial_r^p u \right\|_{(\rho'', \sigma')} 
& \leq 
\frac{p!}{(\rho' - \rho'')^p} 
\left\| u \right\|_{(\rho', \sigma')} 
\quad \text{if } \rho'' < \rho', \; \sigma'' = \sigma', \label{eq:cauchytwo} \\
\left\| \partial_\theta^q u \right\|_{(\rho', \sigma'')} 
& \leq 
\frac{q!}{(\sigma' - \sigma'')^q} 
\left\| u \right\|_{(\rho', \sigma')} 
\quad \text{if } \rho'' = \rho', \; \sigma'' < \sigma', \label{eq:cauchythree}
\end{align}
i.e., depending if the domain is shrank in both variables, in the radial direction only or in the angular direction only, respectively. 
\end{prop} 
\proof Omissis. \endproof
The bounds above will be used in the context of a system of domain restrictions. For this purpose, we define  $K \in \mathbb{N}$ to be fixed later. Then, for any $ k = 1,\ldots, K $, we consider the sequence of nested domains $D_{(1 - d_k)(\rho, \sigma)}$ where the ``shrinking scheme'' is given by
\[
d_k := k/(2 K) \mx{,} 
\]
see e.g. \cite{gior03}. We assume as inductive hypothesis that, for all $ j < k $,
\beq{eq:inductive}
\norm{u_j}{(\rho_j,\sigma_j)} \leq U_j, \qquad \norm{w_j}{(\rho_j,\sigma_j)} \leq W_j
\eeq
and aim to determine what is the values of $U_k$ and $W_k$ as a function of the previously determined $U_0,\ldots, U_{k-1}$ and $W_1,\ldots, W_{k-1}$.\\
We assume that, for any $M \in \mathbb{N}$ fixed and $n \in \mathbb{N}$, the sequence $\{a_{\cdot,m}\}_m$ is chosen in such a way  
\beq{eq:decayanm}
\sum_{m=1}^M |a_{n,m}^{(k)}|\leq \Gamma \exp(-n \sigma) \mx{,}
\eeq
for some $\Gamma >0$.
\subsection{Estimates on the known terms}
In this section we aim to derive quantitative bounds on those object which are known at each $k-$th stage of the hierarchy, i.e. the nonlinear forcing term $\mathcal{F}^{(k)}(r,\theta)$, as well as the Dirichlet and Neumann boundary contributions.   
\subsubsection{Nonlinear Forcing}
We are now concerned with an estimate for the terms $|\xi_0^{(k)}(r)|_{\rho_k}, |\xi_n^{(k)}(r)|_{\rho_k}$ and $|\eta_n^{(k)}(r)|_{\rho_k}$.\\
Recalling (\ref{eq:source}), it is immediate to bound
\[
\norm{\mathcal{F}^{(k)}(r,\theta)}{(\rho_{k-1},\sigma_{k-1})} \leq \sum_{j=1}^{k-1} \norm{w_j(r,\theta)}{(\rho_j,\sigma_j)} \norm{u_{k-j}(r,\theta)}{(\rho_{k-j},\sigma_{k-j})} \leq \sum_{j=1}^{k-1} W_j U_{k-j} \mx{,}
\]
where the first inequality holds as either $D_j$ or $D_{k-j}$ are larger than $D_k$, being $j,k-j<k$. The last inequality implies, see Rem. \ref{rem:two}, 
\beq{eq:nonlinearbound}
|\xi_0^{(k)}(r)|_{\rho_{k-1}} \leq \sum_{j=1}^{k-1} W_j U_{k-j}, \qquad |\xi_n^{(k)}(r)|_{\rho_{k-1}},|\eta_n^{(k)}(r)|_{\rho_{k-1}} \leq \left[ \sum_{j=1}^{k-1} W_j U_{k-j} \right]e^{-n \sigma_{k-1}} \mx{.}
\eeq

\subsubsection{Dirichlet Boundary condition}
Let us now deal with a bound for the ``Dirichlet target function'' $\Phi^{(k)}(\theta)$, defined in (\ref{eq:dirichlet}), which will produce an estimate for its Fourier coefficients $|\widetilde{A}_0^{(k)\prime}|_{\rho_k}, |\widetilde{A}_n^{(k)\prime}|_{\rho_k}$ and $|\widetilde{B}_n^{(k)\prime}|_{\rho_k}$. Assumption (\ref{eq:g}) will be used throughout this section as well as the next one.\\ 
First of all, from (\ref{eq:cauchytwo}), by setting $\rho''=:\rho_{k-\frac{1}{2}} < \rho_{k-j}=: \rho'$, we get 
\beq{eq:dirichletboundone}
\norm{\partial_r^j u_{k-j}}{(\rho_{k-\frac{1}{2}},\sigma_{k-j})} 
\leq j! \frac{ \norm{u_{k-j}}{(\rho_{k-j},\sigma_{k-j})}}{(\rho_{k-j}-\rho_{k-\frac{1}{2}})^j}
\leq j! \frac{U_{k-j}}{(d_{k-\frac{1}{2}}-d_{k-j})^j \rho^j}
 \leq j! \left( \frac{4 K}{\rho j} \right)^j U_{k-j} \mx{.}
\eeq
where we have used $d_{k-\frac{1}{2}}-d_{k-j}=(2j-1)/4K \geq j/(4 K)$. Hence, using an arbitrary restriction in the angular direction by monotony, we have
\[
\norm{\Phi^{(k)}}{(\rho_{k-\frac{1}{2}},\sigma_{k-\frac{1}{2}})}  \leq \sum_{j=1}^k \frac{G^j}{j!} j! \left( \frac{4 K}{\rho j} \right)^j U_{k-j}  \leq \mathcal{M} \sum_{j=1}^k U_{k-j} \mx{.}
\]
Where we have used the fact that the function $j \rightarrow (A/j)^j$ has a maximum in $j^*=A/e$, so
\beq{eq:expinequality}
(A/j)^j \leq (A/(A/e))^{A/e} = e^{A/e} \mx{,}
\eeq
then we have defined 
\beq{eq:m}
\left( \frac{4 K}{\rho j} \right)^j  \leq e^{\frac{4 K}{e \rho }}=: \mathcal{M} \mx{.}
\eeq
As a consequence,
\beq{eq:anbound}
|\widetilde{A}_0^{(k)}|_{\rho_{k-\frac{1}{2}}} \leq \mathcal{M} \sum_{j=1}^{k} U_{k-j}, \qquad |\widetilde{A}_n^{(k)}|_{\rho_{k-\frac{1}{2}}},|\widetilde{B}_n^{(k)}|_{\rho_{k-\frac{1}{2}}} \leq \left[ \mathcal{M} \sum_{j=1}^{k} U_{k-j} \right]e^{-n \sigma_{k-\frac{1}{2}}} \mx{.}
\eeq

\subsubsection{Neumann Boundary condition}\label{subsec:nuemannboundary}
Similarly, a bound for the ``Neumann target function'' $\Psi^{(k)}(\theta)$, will produce an estimate on the Fourier coefficients $|\widetilde{A}_0^{(k)\prime}|_{\rho_k}, |\widetilde{A}_n^{(k)\prime}|_{\rho_k}$ and $|\widetilde{B}_n^{(k)\prime}|_{\rho_k}$. Let us recall (\ref{eq:psi}), which we rewrite below for convenience
\begin{align}
\Psi^{(k)}(\theta)
:={}& g'\sum_{j=0}^{k-1}\frac{g^{k-1-j}}{(k-1-j)!}
\partial_r^{k-1-j}\partial_\theta u_j(1,\theta)-\sum_{j=0}^{k-1}\frac{g^{k-j}}{(k-j)!}
\partial_r^{k-j+1}u_j(1,\theta) \nonumber \\
&-2\sum_{j=0}^{k-1}\frac{g^{k-j}}{(k-1-j)!}
\partial_r^{k-j}u_j(1,\theta)-\sum_{j=0}^{k-2}\frac{g^{k-j}}{(k-2-j)!}
\partial_r^{k-1-j}u_j(1,\theta) -C C_k(\theta).\\
& =: (I)+(II)+\ldots+(V). \label{eq:neumannnotation}
\end{align}
With the aim to bound $\| \Psi^{(k)}\|_{(\rho_{k-\frac12},\sigma_{k-\frac12})}$, we derive some estimates of each addend appearing in the latter, after having used standard Cauchy-Schwartz and triangle inequalities, as well as (\ref{eq:g}).\\
Let us now focus on the first addend $(I)$. By using (\ref{eq:cauchyone}), we obtain:
\begin{align*}
\frac{1}{(k-j-1)!}
\left\| \partial_r^{k - 1 - j} \partial_\theta u_j \right\|_{(\rho_{k - \frac12}, \sigma_{k - \frac12})}
& \leq
\frac{ (k - 1 - j)! }{ ( \rho_j - \rho_{k - \frac12} )^{k - 1 - j} ( \sigma_j - \sigma_{k - \frac12})}
\left\| u_j \right\|_{(\rho_j, \sigma_j)}\\
& \leq \frac{ (4K)^{ k - j } }{ ( k - j )^{ k - j } \rho^{ k - 1 - j } \sigma }
\cdot
\norm{u_j}{(\rho_j, \sigma_j)},
\end{align*}
derived by computing the quantities 
\beq{eq:differences}
\rho_j - \rho_{k - \frac12} = (2K)^{-1}(k - j - 1/2) \rho,
\quad
\sigma_j - \sigma_{k - \frac12} = (2K)^{-1}( k - j - 1/2) \sigma,
\eeq
then bounding $(k-j-1/2)^{k-j} \geq [(k-j)/2]^{k-j}$. As a consequence, the complete bound for the sum is:
\[
\sum_{j=0}^{k-1}
\frac{ (4K)^{ k - j } }{ ( k - j )^{ k - j } \rho^{ k - 1 - j } \sigma }
\left\| u_j \right\|_{(\rho_j, \sigma_j)} \leq \frac{\rho}{\sigma}  \sum_{j=0}^{k-1} \left( \frac{4K}{\rho(k-j)}\right)^{k-j} U_j \mx{.}
\]
This bound yields, by using (\ref{eq:expinequality}) once more and the inductive hypothesis, the following estimate
\beq{eq:neu_i}
\norm{(I)}{(\rho_{k - \frac12},\sigma_{k- \frac12})} \leq \frac{\rho}{\sigma} \mathcal{M}  \sum_{j=0}^{k-1} G^{k-j}U_{j} \leq \frac{\rho}{\sigma} \mathcal{M} G  \sum_{j=0}^{k-1} U_{j} \mx{.}
\eeq
As for the second term, (\ref{eq:cauchythree}) yields
\begin{align*}
\frac{1}{(k-j)!}
\norm{\partial_r^{ k - j + 1 } u_j}{(\rho_{k-\frac12}, \sigma_{k-\frac12})}
& \leq
\frac{ ( k - j + 1 ) }{ ( \rho_j - \rho_{k-\frac12} )^{ k - j + 1 } }
\left\| u_j \right\|_{(\rho_j, \sigma_j)}\\
& \leq
\frac{(k - j + 1) (4K)^{k - j + 1} }{ (k - j)^{k - j + 1} \rho^{k - j + 1} }
\left\| u_j \right\|_{(\rho_j, \sigma_j)} \\
& \leq (K+1)  \left( \frac{4 K}{\rho (k-j)}\right)^{k-j+1} U_j \mx{.}
\end{align*}
At this point, by setting $m=k-j$, we observe that 
\[
(A/m)^{m+1} \leq (A/m) e^{A/e} \leq A e^{A/e}  \leq e [(1/e) e^{2A/e}] = e^{2A/e} \mx{.} 
\]
where we have used (\ref{eq:expinequality}) and that $x \exp(x) \leq \exp(2x)$, for all $x \geq 0$. This leads to 
\beq{eq:neu_ii}
\norm{(II)}{(\rho_{k-\frac12}, \sigma_{k-\frac12})} \leq  \mathcal{M}^2 G (K+1)
\sum_{j=0}^{k-1} U_j \mx{.}
\eeq
The addends $(III)$ and $(IV)$ are treated analogously, obtaining
\begin{align}
\norm{(III)}{(\rho_{k-\frac12}, \sigma_{k-\frac12})} & \leq 
2K \mathcal{M} G 
\sum_{j=0}^{k-1} U_j \mx{,} \label{eq:neu_iii}\\
\norm{(IV)}{(\rho_{k-\frac12}, \sigma_{k-\frac12})} & \leq K \mathcal{M} G 
\sum_{j=0}^{k-2} U_j \leq K \mathcal{M} G 
\sum_{j=0}^{k-1} U_j\mx{.} \label{eq:neu_iv}
\end{align}
Let us finally address the problem of bounding the term $(V)$. To this end, let us recall the standard bounds 
\[
\Bigl|\binom{\tfrac12}{m}\Bigr|\le 1,\qquad \binom{m}{r}\le 2^m.
\]
By using $G=1/4$ directly, we get
\begin{align*}
\left\|
\binom{\tfrac12}{m}\binom{m}{k-m}(2g)^{2m-k}(g^2+(g')^2)^{k-m}
\right\|_{(\rho_{k-\frac12},\sigma_{k-\frac12})}
& \le
2^m\Bigl(\frac14\Bigr)^{2m-k}\Bigl(\frac1{32}\Bigr)^{k-m}\\
& = 2^{2m-3k} \mx{,}\\
\left\|
g\binom{\tfrac12}{m}\binom{m}{k-1-m}(2g)^{2m-(k-1)}(g^2+(g')^2)^{k-1-m}
\right\|_{(\rho_{k-\frac12},\sigma_{k-\frac12})}
& \le
2^m\Bigl(\frac18\Bigr)\Bigl(\frac14\Bigr)^{2m-k+1}\Bigl(\frac1{32}\Bigr)^{k-1-m}\\
& = 2^{2m-3k} \mx{.}  
\end{align*}
We notice that, for all $m \leq k$, $2^{2m-3k}=2^{-k}2^{2(m-k)}\le 2^{-k}$, hence
\begin{align*}
\sum_{m=\lceil k/2\rceil}^{k}2^{2m-3k}
& \le \sum_{m=\lceil k/2\rceil}^{k}2^{-k}
= \bigl(k-\lceil k/2\rceil+1\bigr) 2^{-k}
\le (k+1) 2^{-k} \mx{,} \\
\sum_{m=\lceil (k-1)/2\rceil}^{k-1}2^{2m-3k}
& \le \sum_{m=\lceil (k-1)/2\rceil}^{k-1}2^{-k}
= \bigl(k-\lceil (k-1)/2\rceil\bigr) 2^{-k}
\le k 2^{-k} \mx{.}  
\end{align*}
As a consequence, we obtain 
\[
\|C_k\|_{(\rho_{k-\frac12},\sigma_{k-\frac12})}
\le
\sum_{m=\lceil k/2\rceil}^{k}2^{2m-3k}
+
\sum_{m=\lceil (k-1)/2\rceil}^{k-1}2^{2m-3k}
\le
(k+1)2^{-k}+k2^{-k}
\le
(2k+1)2^{-k}, 
\]
implying
\beq{eq:neu_v}
\|C C_k\|_{(\rho_{k-\frac12},\sigma_{k-\frac12})}
\le
C(2k+1) 2^{-k} = (2 k+1)2^{-k-1} \leq 1 \mx{.}
\eeq
In conclusion, by collecting \eqref{eq:neu_i}, \eqref{eq:neu_ii}, \eqref{eq:neu_iii}, \eqref{eq:neu_iv} and \eqref{eq:neu_v}, then using 
\[
\mathcal{M}[\rho/\sigma+(K+1)\ml{M}+(5/4)K]\sum_{j=0}^{k-1}G^{k-j} \leq G \mathcal{M}[\rho/\sigma+4 \mathcal{M}(K+1)] \leq \mathcal{M}[\rho/\sigma+\mathcal{M}(K+1)] \mx{,}
\]
recall (\ref{eq:g}), we get 
\begin{align}
|\widetilde{A}_0^{(k)\prime}|_{\rho_{k-\frac12}} & \leq 1 + \mathcal{M} \left[\frac{\rho}{\sigma} + \mathcal{M} (K+1)\right] \sum_{j=0}^{k-1} U_j \mx{,} \label{eq:azpbound}
\\
|\widetilde{A}_n^{(k)\prime}|_{\rho_{k-\frac12}}, \; |\widetilde{B}_n^{(k)\prime}|_{\rho_{k-\frac12}}  & \leq
\left\{1 + \mathcal{M} \left[\frac{\rho}{\sigma} + \mathcal{M} (K+1)\right] \sum_{j=0}^{k-1} U_j \right\} e^{-n \sigma_{k-\frac12}} \mx{.} \label{eq:anpbound}
\end{align}

\subsection{Estimate for $w_k(r,\theta)$}
The aim of this section is to determine a bound for the following objects
\[
|a_0^{(k)}(r)|_{\rho_k}, \qquad \sum_{n=1}^{\infty} |a_n^{(k)}(r)|_{\rho_k} e^{n \sigma_k}, \qquad \sum_{n=1}^{\infty} |b_n^{(k)}(r)|_{\rho_k} e^{n \sigma_k} \mx{.}
\]
Let us start with some preliminary estimates
\begin{prop}\label{prop:inequalities} The following inequalities hold
\begin{align}
\sum_{n=1}^{\infty} (4 + n) e^{-n \sigma_{k-1}} e^{n \sigma_k} & \leq 
2 K (e\sigma)^{-1} + 4 \label{eq:prelone}, \\
\sum_{n=1}^{\infty} (2 + n)(4 + n) e^{-n \sigma_{k - \frac12}} e^{n \sigma_k} & \leq 
64 K^2 (\sigma e)^{-2} + 24 K (\sigma e)^{-1} + 8, \label{eq:preltwo} \\
\sum_{n=1}^{\infty} n (2 + n)(4 + n) e^{-n \sigma_{k - \frac12}} e^{n \sigma_k} & \leq  
1728 K^3 (\sigma e)^{-3}
+ 384 K^2 (\sigma e)^{-2} + 32 K (\sigma e)^{-1}. \label{eq:prelthree}
\end{align}
\end{prop}
\proof Let us show how to proceed with (\ref{eq:prelthree}), being the others similar. By setting $a:=\sigma_{k - 1} - \sigma_k=\sigma/(4K)$, it is sufficient to write
\[
\sum_{n=1}^{\infty} (n^3 + 6n^2 + 8n) e^{-n a} = \sum_{n=1}^{\infty} n^3 e^{-n a} + 6 \sum_{n=1}^{\infty} n^2 e^{-n a} + 8 \sum_{n=1}^{\infty} n e^{-n a} \mx{,}
\]
then use, for each addend, the inequality
\[
n^q e^{-a n}  \leq  \left( q/a \right)^q e^{-q} \mx{,}
\]
whose proof is analogous to the one for (\ref{eq:expinequality}).
\endproof
Let us now recall the expressions (\ref{eq:anseries}) and the splitting
\[
a_n^{(k)}(r)=a_{n,0}^{(k)}+\sum_{m=1}^{M} a_{n,m}^{(k)}, \qquad 
b_n^{(k)}(r)=b_{n,0}^{(k)}+\sum_{m=1}^{M} b_{n,m}^{(k)}, \qquad 
\]
for $n \geq 0$ and $n \geq 1$, respectively, understood $\sum_{m=1}^M \cdot =0$ if $M=0$, where the first term of the r.h.s. has been constructed via (\ref{eq:azz}), (\ref{eq:anm}) and (\ref{eq:bnm}), the other terms being arbitrary.\\Let us firstly find a bound for (\ref{eq:anm}). For this purpose, by using in particular the monotonicity $|f(r)|_{\rho_{k}} \leq |f(r)|_{\rho_{k-\frac12}}$, we can proceed as follows
\begin{align*}
|a_n^{(k)}(r)|_{\rho_k}  \leq & |a_{n,0}^{(k)}|_{\rho_{k-\frac12}} + \sum_{m=1}^{M} |a_{n,m}^{(k)}(r)|_{\rho_{k-\frac12}} \\ 
\leq &  2(2+n)(4+n) |\widetilde{A}_n^{(k)\prime}|_{\rho_{k-\frac12}} + 2 n (2+n)(4+n) |\widetilde{A}_n^{(k)}|_{\rho_{k-\frac12}}\\
+& 2(2+n)(4+n) \left| \int_0^1 s^{1+n}ds \right| |\xi_{n}^{(k)}(r)|_{\rho_{k-\frac12}}\\
+& 2(2+n)(4+n) \cdot \frac{1}{2} \sum_{m=1}^{M} \frac{|a_{n,m}^{(k)}|_{\rho_{k-\frac12}}}{(2+n+m)(4+n+m)} + \sum_{m=1}^{M} |a_{n,m}^{(k)}|_{\rho_{k-\frac12}} \\
\leq & 2(2+n)(4+n) \left\{1 + \mathcal{M} \left[\frac{\rho}{\sigma} + \mathcal{M} (K+1)\right] \sum_{j=0}^{k-1} U_j \right\} e^{-n \sigma_{k-\frac12}} \\
+& 2 n (2+n)(4+n) \left[ \mathcal{M} \sum_{j=1}^{k} U_{k-j} \right]e^{-n \sigma_{k-\frac{1}{2}}} \\
+&  2 (n+4) \left[ \sum_{j=1}^{k-1} W_j U_{k-j} \right]e^{-n \sigma_{k-1}} + 2 \Gamma e^{-n \sigma} \mx{,}
\end{align*}
where (\ref{eq:anpbound}), (\ref{eq:anbound}), (\ref{eq:nonlinearbound}), the trivial bound 
  \[
2 (n+2) (n+4) \cdot \frac{1}{2} \sum_{m=1}^{M} \frac{|a_{n,m}^{(k)}|}{(2+n+m)(4+n+m)} \leq \sum_{m=1}^{M} |a_{n,m}^{(k)}| \mx{,}
\]
then $\int_0^1 s^{1+n}ds=(n+2)^{-1}$ and finally assumption (\ref{eq:decayanm}) have been used. \\
It is now sufficient to multiply by $e^{n \sigma_k}$ and use the inequalities found in Prop. \ref{prop:inequalities}, to get
\begin{align}
|a_n^{(k)}(r)|_{\rho_k}e^{n \sigma_k} \leq &  \left\{1 + \mathcal{M} \left[\frac{\rho}{\sigma} + \mathcal{M} (K+1)\right] \sum_{j=0}^{k-1} U_j \right\} \left\{ \frac{128 K^2 }{\sigma^2 e^2} + \frac{48 K }{\sigma e} + 16 \right\} \nonumber \\
+ & \left[ \mathcal{M} \sum_{j=1}^{k} U_{k-j} \right] \left\{\frac{3456 K^3 }{\sigma^3 e^3}
+ \frac{768 K^2 }{\sigma^2 e^2} + \frac{64 K }{\sigma e} \right\} \nonumber \\
+ & \left[ \sum_{j=1}^{k-1} W_j U_{k-j} \right] \left\{\frac{4K}{e\sigma}+8 \right\} + 
4 \Gamma  \nonumber \\
=: & \frac{1}{3} \left[\mathcal{W}_1+ \mathcal{W}_2 \sum_{j=0}^{k-1} U_j + \mathcal{W}_3 \sum_{j=1}^{k-1} W_j U_{k-j} \right]\mx{,} \label{eq:ankrbig}
\end{align}
where we have observed, in order to obtain $4 \Gamma$ in the previous bound, that 
\beq{eq:sumseries}
\sum_{n=1}^{\infty} e^{-n \frac{k \sigma}{2 K}} = \left(1- e^{-\frac{k \sigma}{2 K}}\right)^{-1} e^{-\frac{k\sigma}{2 K}} \in (1,2) \mx{,}
\eeq
for all $\sigma \leq 2 K \log(2)$ i.e. for all $\sigma \leq 1$ by assumption on $K$. Furthermore, we have used that $a \sum_{j=0}^{k-1} U_j + b \sum_{j=1}^{k} U_{k-j} = (a + b) \sum_{j=0}^{k-1} U_j $ and introduced the definitions
\begin{align}
\frac{\mathcal{W}_1}{3} &= \frac{128 K^2 e^{-2}}{\sigma^2} + \frac{48 K e^{-1}}{\sigma} + 16
+ 4 \Gamma, \label{eq:wuno}\\
\frac{\mathcal{W}_2}{3} &= \mathcal{M} \left[\frac{\rho}{\sigma} + \mathcal{M}(K+1)\right] \left( \frac{128 K^2 e^{-2}}{\sigma^2} + \frac{48 K e^{-1}}{\sigma} + 16 \right) \nonumber \\
&+ \mathcal{M} \left( \frac{3456 K^3 e^{-3}}{\sigma^3} + \frac{768 K^2 e^{-2}}{\sigma^2} + \frac{64 K e^{-1}}{\sigma} \right),  \label{eq:wdue}\\
\frac{\mathcal{W}_3}{3} &= \frac{4 K}{e \sigma} + 8. \label{eq:wtre}
\end{align}
Analogously, 
\[
|b_n^{(k)}(r)|_{\rho_k}e^{n \sigma_k} \leq \frac{1}{3} \left[ \mathcal{W}_1+ \mathcal{W}_2 \sum_{j=0}^{k-1} U_j + \mathcal{W}_3 \sum_{j=1}^{k-1} W_j U_{k-j} \right] \mx{.}
\]
On the other hand, the average value is given by
\begin{align*}
|a_0^{(k)}(r)|_{\rho_k}  \leq & |a_{0,0}^{(k)}|_{\rho_{k-\frac12}} + \sum_{m=1}^{M} |a_{0,m}^{(k)}(r)|_{\rho_{k-\frac12}} \\ 
\leq & 16 \left[ |\widetilde{A}_0^{(k)\prime}|_{\rho_{k-\frac12}}  + \left| \int_0^1 s ds \right| |\xi_{0}^{(k)}(r)|_{\rho_{k-\frac12}} + \frac{1}{2} \sum_{m=1}^{M} \frac{|a_{0,m}^{(k)}|_{\rho_{k-\frac12}}}{(2+m)(4+m)} \right] + \sum_{m=1}^{M} |a_{0,m}^{(k)}|_{\rho_{k-\frac12}}  \\
\leq &  16 \left\{1 + \mathcal{M} \left[\frac{\rho}{\sigma} + \mathcal{M} (K+1)\right] \right\} \sum_{j=0}^{k-1} U_j + 8 \sum_{j=1}^{k-1} W_j U_{k-j} + 2 \Gamma \mx{,}
\end{align*}
having used 
\[
8 \sum_{m=1}^{M} \frac{|a_{0,m}^{(k)}|_{\rho_{k-\frac12}}}{(2+m)(4+m)} \leq 
\sum_{m=1}^{M} |a_{0,m}^{(k)}(r)|_{\rho_{k-\frac12}} \leq \Gamma \mx{.}
\]
After a trivial overestimation of the constants at hand, this leads to
\beq{eq:azkr}
|a_0^{(k)}(r)|_{\rho_k} \leq \frac{1}{3} \left[ \mathcal{W}_1+ \mathcal{W}_2 \sum_{j=0}^{k-1} U_j + \mathcal{W}_3 \sum_{j=1}^{k-1} W_j U_{k-j} \right] \mx{.}
\eeq
Hence, in conclusion,
\begin{align*}
\norm{w_k}{(\rho_k,\sigma_k)} & = |a_0^{(k)}(r)|_{\rho_k} + 
\sum_{n=1}^{\infty} |a_n^{(k)}(r)|_{\rho_k}e^{n \sigma_k} + 
\sum_{n=1}^{\infty} |b_n^{(k)}(r)|_{\rho_k}e^{n \sigma_k}\\ 
& \leq \mathcal{W}_1+ \mathcal{W}_2 \sum_{j=0}^{k-1} U_j + \mathcal{W}_3 \sum_{j=1}^{k-1} W_j U_{k-j} \mx{,}
\end{align*}
which leads to the recurrence equation
\beq{eq:recurrencewk}
W_k=\mathcal{W}_1+ \mathcal{W}_2 \sum_{j=0}^{k-1} U_j + \mathcal{W}_3 \sum_{j=1}^{k-1} W_j U_{k-j} \mx{.}
\eeq
\subsection{Estimate for $u_k(r,\theta)$}
The aim of this section is to find a bound for 
\[
|A_0^{(k)}(r)|_{\rho_{k}}, \quad
\sum_{n=1}^{\infty} |A_n^{(k)}(r)|_{\rho_{k}} e^{n \sigma_{k}}, \quad 
\sum_{n=1}^{\infty} |B_n^{(k)}(r)|_{\rho_{k}} e^{n \sigma_{k}} \mx{.}
\]
First of all, from (\ref{eq:azeropre}),
\[
|A_0^{(k)}(r)|_{\rho_{k}} \leq |\widetilde{A}_0^{(k)}| + |\mathfrak{f}_0^{(k)}(r)|_{\rho_{k}} \int_r^1 s^{-1} \int_0^s s' ds' \leq |\widetilde{A}_0^{(k)}| + |\mathfrak{f}_0^{(k)}(r)|_{\rho_{k}} (1-r) \leq |\widetilde{A}_0^{(k)}| + |\mathfrak{f}_0^{(k)}(r)|_{\rho_{k}} \mx{.}
\]
Let us now preliminarily observe that, for all $n \geq 1$,
\begin{align*}
\left| r^n \int_r^1 s^{1-n}ds \right| & \leq r^n \int_r^1 1^{1-n} ds \leq r^n (1-r) \leq r^n \leq 1, \\
\left| r^{-n} \int_0^r s^{1+n}ds \right| & \leq r^{-n} \int_0^r r^{n+1} ds \leq r^{-n} \int_0^r r^{n-1} ds = n^{-1} r^{-n}  r^n \leq 1 \mx{.}
\end{align*}
Then, from (\ref{eq:an}) and (\ref{eq:bn}), by using the two bounds above
\begin{align*}
|A_n^{(k)}(r)|_{\rho_{k}} \leq |C_n^{(k)}| + (2 n)^{-1}  |\mathfrak{f}_n^{(k)}(r)|_{\rho_{k}} + (2 n)^{-1}  |\mathfrak{f}_n^{(k)}(r)|_{\rho_{k}} \leq |C_n^{(k)}| + |\mathfrak{f}_n^{(k)}(r)|_{\rho_{k}} \mx{.}
\end{align*}
Similarly, we find
\[
|B_n^{(k)}(r)|_{\rho_{k}} \leq |D_n^{(k)}| + |\mathfrak{g}_n^{(k)}(r)|_{\rho_{k}} \mx{.}
\]
On the other hand, by recalling the definitions (\ref{eq:cean}) and (\ref{eq:cebn}), 
\[
|\mathfrak{f}_n^{(k)}(r)|_{\rho_{k}} \leq (1/4) |a_n^{(k)}(r)|_{\rho_{k}}  + |\xi_n^{(k)}(r)|_{\rho_{k}}, \qquad 
|\mathfrak{g}_n^{(k)}(r)|_{\rho_{k}} \leq (1/4) |b_n^{(k)}(r)|_{\rho_{k}}  + |\eta_n^{(k)}(r)|_{\rho_{k}} \mx{.}
\]
As a consequence, by using the values determined in (\ref{eq:cndn}) for $C_n^{(k)}$ and $D_n^{(k)}$, we obtain for all $n \geq 1$, 
\begin{align*}
|A_0^{(k)}(r)|_{\rho_{k}} & \leq |\widetilde{A}_0^{(k)}| +  |a_0^{(k)}(r)|_{\rho_{k}} +  |\xi_0^{(k)}(r)|_{\rho_{k}},  \\ 
|A_n^{(k)}(r)|_{\rho_{k}} & \leq |\widetilde{A}_n^{(k)}| + |\widetilde{A}_n^{(k) \prime}| +   |a_n^{(k)}(r)|_{\rho_{k}} +  |\xi_n^{(k)}(r)|_{\rho_{k}}, \\
|B_n^{(k)}(r)|_{\rho_{k}} & \leq |\widetilde{B}_n^{(k)}| + |\widetilde{B}_n^{(k) \prime}| +   |b_n^{(k)}(r)|_{\rho_{k}} +  |\eta_n^{(k)}(r)|_{\rho_{k}} \mx{.}
\end{align*}
Let us start by finding a bound for the first of them. Recalling the domain monotony property, then (\ref{eq:anbound}), (\ref{eq:azkr}) and the first of (\ref{eq:nonlinearbound}),
\begin{align}
|A_0^{(k)}(r)|_{\rho_{k}} & \leq |\widetilde{A}_0^{(k)}| +  |a_0^{(k)}(r)|_{\rho_{k-\frac12}} +  |\xi_0^{(k)}(r)|_{\rho_{k-\frac12}} \nonumber \\
& \leq \mathcal{M} \sum_{j=1}^k U_{k-j} +
\frac{1}{3} \left[ \mathcal{W}_1+ \mathcal{W}_2 \sum_{j=0}^{k-1} U_j + \mathcal{W}_3 \sum_{j=1}^{k-1} W_j U_{k-j} \right]  + \sum_{j=1}^{k-1} W_j U_{k-j} \nonumber \\
& \leq \frac{1}{3} \left[ \mathcal{W}_1+ (\mathcal{W}_2 + 3 \mathcal{M})\sum_{j=0}^{k-1} U_j + (\mathcal{W}_3 + 3) \sum_{j=1}^{k-1} W_j U_{k-j} \right] \label{eq:capitalazkr} \mx{.}
\end{align}
Then we can proceed with the second one, which requires once more (\ref{eq:anbound}), then (\ref{eq:anpbound}), as well as the second of (\ref{eq:nonlinearbound})
\begin{align*}
|A_n^{(k)}(r)|_{\rho_{k}} & \leq \left[  \mathcal{M} \sum_{j=1}^k U_{k-j}  \right] e^{-n \sigma_{k - \frac12}} + \left\{1 + \mathcal{M} \left[\frac{\rho}{\sigma} + \mathcal{M} (K+1)\right] \sum_{j=0}^{k-1} U_j \right\} e^{-n \sigma_{k - \frac12}}\\
& + |a_n^{(k)}(r)|_{\rho_{k-\frac12}}+ \left[ \sum_{j=1}^{k-1} W_j U_{k-j}  \right] e^{-n \sigma_{k - \frac12}} \mx{,}
\end{align*}
it is sufficient to multiply both sides by $e^{n \sigma_k}$ and use (\ref{eq:ankrbig}), to get 
\begin{align*}
|A_n^{(k)}(r)|_{\rho_{k}} e^{n \sigma_k} & \leq  2 \left[  \mathcal{M} \sum_{j=1}^k U_{k-j}  
+ 1 + \mathcal{M} \left[\frac{\rho}{\sigma} + \mathcal{M} (K+1)\right] \sum_{j=0}^{k-1} U_j + \sum_{j=1}^{k-1} W_j U_{k-j}  \right] \\
& + \frac{1}{3} \left[ \mathcal{W}_1+ \mathcal{W}_2 \sum_{j=0}^{k-1} U_j + \mathcal{W}_3 \sum_{j=1}^{k-1} W_j U_{k-j} \right]  \\
& =: \frac{1}{3} \left[ \mathcal{U}_1+ \mathcal{U}_2 \sum_{j=0}^{k-1} U_j + \mathcal{U}_3 \sum_{j=1}^{k-1} W_j U_{k-j} \right] \mx{,}
\end{align*}
where, similarly to (\ref{eq:sumseries}), we have used that $\sum_{n=1}^{\infty} \exp(-n \sigma_{k - \frac12} + n \sigma_k )  \in (1,2)$ under the same condition on $\sigma$, and we have defined
\begin{align}
\mathcal{U}_1 & := \mathcal{W}_1 + 6 \label{eq:uuno} \mx{,}\\
\mathcal{U}_2 & := \mathcal{W}_2 + 6 \mathcal{M}  (1+ \rho^{-1}\sigma + \mathcal{M}(K+1)) \label{eq:udue} \mx{,}\\
\mathcal{U}_3 & := \mathcal{W}_3 + 6 \label{eq:utre} \mx{.}
\end{align}
Similarly, 
\[
|B_n^{(k)}(r)|_{\rho_{k}} e^{n \sigma_k} \leq \frac{1}{3} \left[ \mathcal{U}_1+ \mathcal{U}_2 \sum_{j=0}^{k-1} U_j + \mathcal{U}_3 \sum_{j=1}^{k-1} W_j U_{k-j} \right] \mx{.}
\]
Now it is evident that 
\[
\mathcal{U}_2 > \mathcal{W}_2+3\mathcal{M}\quad ; \quad \mathcal{U}_3 > \mathcal{W}_3+3 \mx{,}
\]
hence we can replace the estimate on (\ref{eq:capitalazkr}) with 
\[
|A_0^{(k)}(r)|_{\rho_{k}}  \leq \frac{1}{3} \left[ \mathcal{U}_1+ \mathcal{U}_2 \sum_{j=0}^{k-1} U_j + \mathcal{U}_3 \sum_{j=1}^{k-1} W_j U_{k-j} \right] \mx{.}
\]
In conclusion, we obtain the following bound
\begin{align}
\norm{u_k}{(\rho_k,\sigma_k)} & = |A_0^{(k)}(r)|_{\rho_k} + 
\sum_{n=1}^{\infty} |A_n^{(k)}(r)|_{\rho_k}e^{n \sigma_k} + 
\sum_{n=1}^{\infty} |B_n^{(k)}(r)|_{\rho_k}e^{n \sigma_k} \nonumber \\
& \leq \mathcal{U}_1+ \mathcal{U}_2 \sum_{j=0}^{k-1} U_j + \mathcal{U}_3 \sum_{j=1}^{k-1} W_j U_{k-j} \mx{,}
\end{align}
leading to the second recurrence equation
\beq{eq:recurrenceuk}
U_k=\mathcal{U}_1+ \mathcal{U}_2 \sum_{j=0}^{k-1} U_j + \mathcal{U}_3 \sum_{j=1}^{k-1} W_j U_{k-j} \mx{.}
\eeq

\section{Majorising Sequence Construction}
Let us consider the system of recurrence equations formed by (\ref{eq:recurrencewk}) and (\ref{eq:recurrenceuk}), which are reported below for the reader's convenience
\beq{eq:recurrencesystem}
\left\{
\begin{aligned}
W_k &= \mathcal{W}_1 + \mathcal{W}_2 \sum_{j=0}^{k-1} U_j + \mathcal{W}_3 \sum_{j=1}^{k-1} W_j U_{k-j}\\
U_k &= \mathcal{U}_1 + \mathcal{U}_2 \sum_{j=0}^{k-1} U_j + \mathcal{U}_3 \sum_{j=1}^{k-1} W_j U_{k-j}
\end{aligned}
\right. \mx{,}
\eeq
where $\mathcal{W}_i, \mathcal{U}_i \geq 0$ are constants, see (\ref{eq:wuno})-(\ref{eq:wtre}) and (\ref{eq:uuno})-(\ref{eq:utre}), with $W_0=0$ by construction while we shall set $U_0 := 1$ for simplicity, as described in the following  
\begin{rem}\label{rem:rho} As anticipated in Rem. \ref{rem:sigma}, by choosing $U_0$ we set a bound on $\rho$. More precisely we have $|(1-z^2)/4| \leq 4^{-1}|1-\rho \exp( 2 i \phi)| \leq 4^{-1}(1+\rho^2) \leq U_0$ (here $\phi$ is the anomaly of $z$). In particular, by setting $U_0=1$ we have $\rho \leq \sqrt{3}$. Hence, we can still choose $\rho =1$.
\end{rem}
It is straightforward to check by induction that the new sequence $\{Z_k\}_{k\geq 0}$, defined by
\[
Z_k := \mathcal{Z}_1 + \mathcal{Z}_2 \sum_{j=0}^{k-1} Z_j + \mathcal{Z}_3 \sum_{j=1}^{k-1} Z_j Z_{k-j} \mx{,}
\]
with $Z_0:= \max\{W_0, U_0\} \equiv U_0$ and $\mathcal{Z}_j := \max\{\mathcal{W}_j, \mathcal{U}_j\}$, with $j=1,2,3$, satisfies
\[
W_k \leq Z_k \quad \text{and} \quad U_k \leq Z_k, \qquad \forall k \geq 0 \mx{.}
\]
\begin{rem}\label{rem:zk}
As it is clear from (\ref{eq:uuno}-\ref{eq:utre}), $\mathcal{Z}_j \equiv \mathcal{U}_j$ for all $j=1,2,3$.
\end{rem}
Let us now split $\sum_{j=0}^{k-1} Z_j = Z_0 + \sum_{j=1}^{k-1} Z_j$ so that the recurrence equation becomes:
\[
Z_k = \mathcal{Z}_1 + \mathcal{Z}_2 Z_0 + \sum_{j=1}^{k-1} Z_j (\mathcal{Z}_2 + \mathcal{Z}_3 Z_{k-j}).
\]
As $ Z_{k-j} \geq 1 $, we can bound $\mathcal{Z}_2 + \mathcal{Z}_3 Z_{k-j} \leq (\mathcal{Z}_2 + \mathcal{Z}_3) Z_{k-j}$. Hence
\[
Z_k \leq \mathcal{Z}_1 + \mathcal{Z}_2 Z_0 + (\mathcal{Z}_2 + \mathcal{Z}_3) \sum_{j=1}^{k-1} Z_j Z_{k-j}.
\]
The next step consists in defining a further sequence, $\{\widetilde{Z}_k\}_{k\geq 0}$, majorising $\{Z_k\}$ by construction, whose solution could be explicitly determined. This is exactly the aim of the next
\begin{prop} 
\beq{eq:ab}
a := \mathcal{Z}_1 + \mathcal{Z}_2 Z_0 \quad ; \quad b := \mathcal{Z}_2 + \mathcal{Z}_3 
\eeq
and define
\beq{eq:ztildek}
\widetilde{Z}_k := a + b \sum_{j=1}^{k-1} \widetilde{Z}_j \widetilde{Z}_{k-j}, \quad k \geq 1.
\eeq
with $\widetilde{Z}_0 := Z_0$. Then, for all $k \geq 1$, 
\beq{eq:evolutionzk}
U_k,W_k \leq \widetilde{Z}_k = \frac{a}{k} \binom{2(k-1)}{k-1} (a b)^{k-1}.
\eeq
\end{prop} 
\proof Uses a straightforward generalisation of the well known Catalan numbers sequence. We give here a brief outline of it for the sake of completeness. By using the so-called generating function method, we proceed by multiplying the $k-$th equation of (\ref{eq:ztildek}) by $z^k$ then summing both left and right hand sides. Hence, by defining $\mathcal{G}(z) := \sum_{k=1}^\infty \widetilde{Z}_k z^k$, we find $\mathcal{G}(z) = \sum_{k=1}^\infty a z^k + b \sum_{k=1}^\infty z^k \sum_{j=1}^{k-1} \widetilde{Z}_j \widetilde{Z}_{k-j} $. Subsequently, by computing the first sum and recognising a product \emph{à la Cauchy} in the second one, we get  
\[
\mathcal{G}(z) = a (1-z)^{-1} z + b \mathcal{G}^2(z).
\]
Solving with respect to $\mathcal{G}(z)$, we get
\[
\ml{G}(z) = (2b)^{-1}\left[1 - \sqrt{1 - (1-z)^{-1}(4 a b z)}\right].
\]
The task now consists in expanding the r.h.s. in powers of $z$. For this purpose, let us recall the formula $2^{-1}(1-\sqrt{1-4x})=\sum_{n=0}^{\infty}(1+n)^{-1}\binom{2n}{n}x^{n-1}$, where we set $x:=ab(1-z)^{-1}z$, then expand the latter in powers of $z$. A straightforward manipulation yields (\ref{eq:evolutionzk}). 
\endproof

\begin{ex}[Numerical validation of (\ref{eq:evolutionzk})] Let us choose 
\[
\mathcal{W}_1=1,\quad \mathcal{W}_2=2,\quad \mathcal{W}_3=1, \quad \mathcal{U}_1=1,\quad \mathcal{U}_2=1,\quad \mathcal{U}_3=2.
\]
Thus $a=3$ and $b=4$. The table below shows the first terms of $\{U_k\}$, $\{W_k\}$ and how majorising property $\widetilde{Z}_k$ holds:
\begin{table}[h!]
\centering
\begin{tabular}{|c|c|c|c|}
\hline
$k$ & $U_k$ & $W_k$ & $\widetilde{Z}_k$ \\
\hline
0 & 1 & 0 & 1 \\
1 & 2 & 3 & 3 \\
2 & 16 & 13 & 39 \\
3 & 168 & 113 & 939 \\
4 & 2064 & 1313 & 28623 \\
5 & 27840 & 17953 & 1043649 \\
6 & 408864 & 266753 & 44272779 \\
7 & 6423936 & 4191809 & 2077497615 \\
8 & 107487168 & 70226401 & 107996103879 \\
9 & 1909610496 & 1241897857 & 6198003389695 \\
\hline
\end{tabular}
\caption{Comparison of the sequences $U_k$, $W_k$, and $\widetilde{Z}_k$ for $k = 0, \ldots, 9$}
\end{table}
\end{ex}
\section{Exponential Bound for \texorpdfstring{$a$}{a} and \texorpdfstring{$b$}{b}}
The next step is concerned with the fact that the objects $a$ and $b$ do depend on $K$. Our aim is to give a bound on such a dependence. Recalling (\ref{eq:ab}) and Rem. \ref{rem:zk}, let us write down the following expressions
\begin{align*}
\mathcal{U}_1 &= \frac{384 K^2 e^{-2}}{\sigma^2} + \frac{144 K e^{-1}}{\sigma} + 54 + 12 \Gamma, \\
\mathcal{U}_2 &= 3 \mathcal{M} \left( \frac{\rho}{\sigma} + \mathcal{M}(K+1) \right) \left( \frac{128 K^2 e^{-2}}{\sigma^2} + \frac{48 K e^{-1}}{\sigma} + 16 \right) \\
&\quad + 3 \mathcal{M} \left( \frac{3456 K^3 e^{-3}}{\sigma^3} + \frac{768 K^2 e^{-2}}{\sigma^2} + \frac{64 K e^{-1}}{\sigma} \right) + 6 \mathcal{M} \left( 1 + \frac{\sigma}{\rho} + \mathcal{M}(K+1) \right), \\
\mathcal{U}_3 &= \frac{12 K}{e \sigma} + 30.
\end{align*}
Let be  
\beq{eq:alphadef}
\alpha:=4/(e \rho) \mx{,}
\eeq
in such a way $ \mathcal{M} = e^{\alpha K} $. We can overestimate the expressions for $ a $ and $ b $ by carrying out the factor $ \mathcal{M}^2 = e^{2\alpha K} $ and setting $ \mathcal{M} = 1 $ inside the polynomials. This gives:
\beq{eq:abdef}
a \leq \mathcal{M}^2  P_a(K), \qquad b \leq \mathcal{M}^2 P_b(K),
\eeq
where the polynomials $\mathcal{P}_a(K)$ and $\mathcal{P}_b(K)$ are written in terms of powers of $K$ as follows 
\begin{align*}
\ml{P}_a(K) &= A_3 K^3 + A_2 K^2 + A_1 K + A_0, \\
\ml{P}_b(K) &= B_3 K^3 + B_2 K^2 + B_1 K + B_0.
\end{align*}
where
\begin{align*}
A_0 &= (2 \rho \sigma)^{-1} [24 \Gamma \rho \sigma + 24 \rho^2 + 30 \rho \sigma + 3 \sigma^2 + 108 \rho \sigma], \\
A_1 &= \left(2 e \sigma^2\right)^{-1} \left[27 e \sigma^2 + 72 \rho + 168 \sigma + 288 \sigma\right], \\
A_2 &= \left(e^2 \sigma^3\right)^{-1} \left[36 e \sigma^2 + 96 \rho + 672 \sigma + 384 \sigma\right], \\
A_3 &= \left(e^3 \sigma^3\right)^{-1} \left[96 e \sigma + 2592\right],
\end{align*}
and
\begin{align*}
B_0 &= \left(\rho \sigma\right)^{-1} \left[48 \rho^2 + 60 \rho \sigma + 6 \sigma^2 + 30 \rho \sigma\right], \\
B_1 &= \left(e \sigma^2\right)^{-1} \left[54 e \sigma^2 + 144 \rho + 336 \sigma + 12 \sigma\right], \\
B_2 &= \left(e^2 \sigma^3\right)^{-1} \left[144 e \sigma^2 + 384 \rho + 2688 \sigma\right], \\
B_3 &= \left(e^3 \sigma^3\right)^{-1} \left[384 e \sigma + 10368\right].
\end{align*}
\begin{prop}
Let 
\[
f(x) := e^{\alpha x}(c_0 + c_1 x + c_2 x^2 + c_3 x^3) \mx{,}
\]
with $ \alpha > 0 $ and real positive constants $ c_0, c_1, c_2, c_3 $. Then, for all $ x \geq 1 $, the following estimate holds:
\beq{eq:propexp}
f(x) \leq \mathscr{B} e^{2\alpha x}, \qquad \mathscr{B} := \left[ 3/(e \alpha) \right]^3 \sum_{j=0}^3 c_j .
\eeq
\end{prop}
\proof
Let us write $f(x)=: e^{\alpha x} \mathcal{P}(x)$. We now aim to bound $ f(x) $ from above by a pure exponential $ e^{\beta x} $ with $ \beta > \alpha $. Choosing $ \beta := 2\alpha $, we have 
$f(x) = e^{2\alpha x} \cdot e^{-\alpha x} \mathcal{P}(x)$, so it is sufficient to bound $e^{-\alpha x} \mathcal{P}(x)$. \\
First of all, as $x \geq 1$, then $\ml{P}(x) \leq (c_0 + c_1 + c_2 + c_3) x^3$. Hence, the analysis is reduced to bound the maximum of the function $ x \mapsto e^{-\alpha x} x^3 $. A straightforward computation shows that this function attains its maximum at $ x = 3/\alpha $, so it is sufficient to substitute this value to obtain bound (\ref{eq:propexp})
\endproof
By the latter and recalling (\ref{eq:abdef}), we have 
\beq{eq:mathscrb}
a,b \leq \mathscr{B} e^{2 \alpha K}, \qquad \mathscr{B}:=(27/8) \Theta \rho^3 \mx{.}
\eeq
where
\beq{eq:capitaltheta}
\begin{aligned}
\Theta & :=\max\left\{ \sum_{j=0}^3 A_j,\ \sum_{j=0}^3 B_j \right\} 
& = 144 +
12 \Gamma + \frac{48\rho}{\sigma} + \frac{6\sigma}{\rho}
+ \frac{144\rho}{e\sigma^2} + \frac{492}{e\sigma}
+ \frac{384\rho}{e^2\sigma^3} + \frac{3072}{e^2\sigma^2}
+ \frac{10368}{e^3\sigma^3}.
\end{aligned}
\eeq

\section{Remainder estimate}
\subsection{Resolvability defect}
Let us recall (\ref{eq:remainder}). By (\ref{eq:evolutionzk}), we have
\begin{align*}
\norm{\mathcal{R}^{[k]}}{(\rho_k,\sigma_k)} & \leq \sum_{n = k+1}^{2k} \mu^n \sum_{j = n-k}^{k} \widetilde{Z}_{n-j} \widetilde{Z}_j \\
& \leq \sum_{n = k+1}^{2k} \mu^n \sum_{j = n-k}^{k}  \frac{a^2}{j(n - j)} \binom{2(j - 1)}{j - 1} \binom{2(n - j - 1)}{n - j - 1} (ab)^{n - 2} \\
& \leq a^2 \sum_{n = k+1}^{2k} \mu^n (2 e a b)^{n-2}\sum_{j = n-k}^{k}  \frac{1}{j(n - j)} \\
& \leq  \sum_{n = k+1}^{+\infty}  (2 e \mu a b)^{n}\\
& \leq (2 e \mu a b)^{k} \mx{,}
\end{align*}
where we have used the Stirling-type bound $\binom{2m}{m} \leq (2e)^m$ to get the third estimate, the inequality $\sum_{j = 1}^{n - 1} \frac{1}{j(n - j)} \leq 1$ to get the fourth one, which can be shown by observing that $f_n(x):=1/[x(n-x)] \leq 1/(n-1)$. Finally, we have used that $2 e b \geq 1$ and the fact that 
\beq{eq:muzero}
2 e \mu a b < 1 \mx{,}
\eeq
which we shall suppose to be true for the moment.\\
We are now interested in bounding the expression
\beq{eq:fk}
F(K) := (2 e a b \mu)^K,
\eeq
where $ a, b \leq \mathscr{B} e^{2\alpha K} $ by (\ref{eq:mathscrb}). Clearly, showing that $F(K) < 1$, implies (\ref{eq:muzero}). By substituting the latter into  $ F(K) $, we have
\[
F(K) \leq \left( 2 e \mathscr{B}^2 \mu \right)^K e^{4\alpha K^2}=:G(K).
\]
The last step consists in minimising this bound with respect to $K$: the minimising $K$ shall be called $K_*$ which is going to be, however, a real number (at least in general). Hence, we shall finalise our choice of $K$ as $K \equiv K_{\text{opt}}$, which is nothing but the lowest integer greater than $K_*$.\\
First of all, we have $
\log G(K) = K \log \big( 2 e \mathscr{B}^2 \mu \big) + 4\alpha K^2$. A derivation w.r.t. $K$ shows that its stationary point $K_*$ is given by 
\beq{eq:kstar}
K_* = -\frac{\log \big( 2 e \mathscr{B}^2 \mu \big)}{8\alpha} \quad \Rightarrow \quad
\log G(K_*) = -\frac{\big[\log( 2 e \mathscr{B}^2 \mu )\big]^2}{16\alpha}.
\eeq
Let us observe that $\log G(K)$ is convex, so $G(K)$ is, hence
\beq{eq:ineqg}
G(K_*) \le G\big(\lceil K_* \rceil\big) \le G(K_*+1).
\eeq
Moreover $\log G(K_*+1) - \log G(K_*) = 4\alpha$, so $G(K_*+1) = e^{4\alpha} G(K_*)$. By using  (\ref{eq:kstar}) in the latter, setting
\beq{eq:kopt}
K_{\text{opt}}:=\lceil K_* \rceil
\eeq
using (\ref{eq:ineqg}), (\ref{eq:fk}) and finally recalling (\ref{eq:alphadef}), one finds
\beq{eq:prefinal}
F(K_{\text{opt}}) \leq G\big( K_{\text{opt}} \big) \leq e^{\frac{16}{e \rho} - \frac{e \rho}{64}\left[\log \big( 2 e \mathscr{B}^2 \mu \big)\right]^2} \mx{.}
\eeq
Hence, $\norm{\mathcal{R}^{[k]}}{(\rho_k,\sigma_k)}$ evaluated at $k=K_{\text{opt}}$, where one sets 
\beq{eq:mathscrc}
\mathscr{C}_0:=\exp(16/(e \rho)), \qquad \mathscr{C}_2:=e \rho/64, \qquad \mathscr{C}_3:=2 e \mathscr{B}^2 \mx{,}
\eeq
is bounded by the r.h.s. of (\ref{eq:prefinal}) in the form
\beq{eq:prefinalone}
\mathscr{C}_0 e^{-\mathscr{C}_2 [\log(\mathscr{C}_3 \mu)]^2}=\mathscr{C}_1 
e^{-\mathscr{C}_2 (\log \mathscr{C}_3)^2} 
e^{-2\mathscr{C}_2 (\log \mathscr{C}_3)(\log \mu)} 
e^{-\mathscr{C}_2 (\log \mu)^2} 
\leq \mathscr{C}_1\mu^{\mathscr{C}_2(\log(1/\mu)-2\log \mathscr{C}_3)} \mx{,}
\eeq
where 
\beq{eq:mathscrcd}
\mathscr{C}_1=\mathscr{C}_0\mathscr{C}_3^{-\mathscr{C}_2\log \mathscr{C}_3} \mx{,}
\eeq
which is exactly the first of (\ref{eq:finalbound}).\\
We are now ready to provide the smallness condition of $\mu$. On one side, in order to get that the exponent in (\ref{eq:finalbound}) is always positive for all $\mu \leq \mu_0$ and the r.h.s. infinitesimal we need to ensure $\mu \leq \mathscr{C}_3^{-2}$. On the other hand, to 
obtain $F(K_{opt})<1$, and hence (\ref{eq:muzero}) satisfied, it is necessary that $\mu < \mathscr{C}_3^{-1} \exp(-32e^{-1}\rho^{-1})$, see (\ref{eq:prefinal}). In conclusion, the smallness condition on $\mu$ is given by
\beq{eq:muzeronew}
\mu_0:=\mathscr{C}_3^{-1}\min\{\mathscr{C}_3^{-1},  2^{-1}\exp(-32 e^{-1}\rho^{-1})\}\mx{.}
\eeq  
The final choice for the constructed (family of) solution(s) will be
\beq{eq:finalvw}
\left(u^*, w^*\right):=\left(u^{[k]},w^{[k]}\right)_{k \equiv K_{\text{opt}}} \mx{.}
\eeq
\subsection{Dirichlet boundary defect}
Let us recall (\ref{eq:dirichletdefect}) and assumption (\ref{eq:g}). By proceeding as in (\ref{eq:dirichletboundone}), after having recalled that 
\beq{eq:normmonotony}
\norm{\cdot}{(\rho_k,\sigma_k)} \leq \norm{\cdot}{(\rho_{k-\frac12},\sigma_{k-\frac12})} \mx{,}
\eeq
by monotony, we obtain 
\begin{align*}
\norm{\mathcal{E}_{D}^{[k]}}{(\rho_k,\sigma_k)} & \le
\sum_{j=k+1}^{\infty}\mu^{j}\sum_{m=j-k}^{j}
\frac{\norm{\partial_r^{m}u_{j-m}}{(\rho_k,\sigma_k)}}{m!} \\
& \leq \sum_{j=k+1}^{\infty}\mu^{j}\sum_{m=j-k}^{j}
\left(\frac{4K}{\rho m}\right)^{m} U_{j-m} \\
&  \leq 2 
\mathcal{M} \mu^{k+1} \sum_{\ell=0}^{k}U_{\ell}.
\end{align*}
The latter has been obtained by introducing $\ml{M}$ as in (\ref{eq:m}), observing that $\sum_{m=j-k}^{j}U_{j-m}=\sum_{\ell=0}^{k}U_{\ell}$ i.e. independent of $j$, then computing $\sum_{j=k+1}^{\infty}\mu^{j}
=(1-\mu)^{-1} \mu^{k+1}\le 2 \mu^{k+1}$ as $\mu \leq 1/2$ by assumption.\\
Now, recalling (\ref{eq:recurrencesystem}), using $\ml{U}_2>1$ then (\ref{eq:evolutionzk}) with $\binom{2(\ell-1)}{\ell-1}\le(2e)^{\ell-1}$ and finally $2 e b \geq 1$, we have
\beq{eq:boundsumul}
\sum_{\ell=0}^{k}U_{\ell} = U_k+ \sum_{\ell=0}^{k-1}U_{\ell} \leq 2 U_k \leq \frac{2a}{k} \binom{2(k-1)}{k-1} (a b)^{k-1} \leq 2 a (2 e a b)^{k-1} \leq 2 (2 e a b)^k \mx{.}
\eeq
Hence, we have by (\ref{eq:fk})
\beq{eq:firstbounded}
\norm{\mathcal{E}_{D}^{[k]}}{(\rho_k,\sigma_k)} \leq 4 \ml{M} \mu F(k) \mx{.}
\eeq
Let us now observe that
\beq{eq:mkstar}
\mathcal{M}(\lceil K_*\rceil) \le \mathcal{M}(K_*+1)
= e^{\alpha(K_*+1)} 
= e^{\alpha}\exp\!\Bigl(\tfrac{1}{8}\log\tfrac{1}{2 e \mathscr{B}^2 \mu}\Bigr)
= e^{\alpha}\Bigl(\tfrac{1}{2 e \mathscr{B}^2 \mu}\Bigr)^{1/8} \le e^{\alpha}\mu^{-1/8} \mx{,}
\eeq
as $2 e \mathscr{B}^2 \ge 1$. Hence, the r.h.s. of (\ref{eq:firstbounded}), evaluated at $k=K \equiv K_{\text{opt}}$, is bounded by $4 \mathcal{M} \mu F(K_{\text{opt}})$. \\
Now it is sufficient to use (\ref{eq:prefinal}) and (\ref{eq:prefinalone}) combined, then define 
\beq{eq:cfour}
\mathscr{C}_4:=4 e^{\frac{4}{\rho \sigma}} \mathscr{C}_3^{-\frac{7}{16}} \mathscr{C}_1 \mx{,}
\eeq
where we have used (\ref{eq:muzero}), to get the second of (\ref{eq:finalbound}).
\subsection{Neumann boundary defect}
As for $\mathcal{E}_{D}^{[k]}(\theta)$, the bounds go along the lines of those carried out in Sec. \ref{subsec:nuemannboundary} as the terms in $\Psi^{(k)}$ differ from those appearing in (\ref{eq:neumanndefect}) for the upper extremum of summation only. As a consequence, similarly to (\ref{eq:neumannnotation}), we denote $\mathcal{E}_{N}^{[k]}(\theta)=\sum_{j=k+1}^{\infty} \mu^j \{(\widetilde{I})+(\widetilde{II})+ \ldots + (\widetilde{V})\}$ where the terms are deduced immediately from (\ref{eq:neumanndefect}). \\
Proceeding as in Sec. \ref{subsec:nuemannboundary}, and using \eqref{eq:normmonotony} once more, we get
\begin{align}
\norm{(\widetilde{I})}{(\rho_{k},\sigma_{k})} & \leq \frac{\rho}{\sigma} \mathcal{M} G  \sum_{h=0}^{k} U_h \mx{,}\\
\norm{(\widetilde{II})}{(\rho_{k}, \sigma_{k})} & \leq  \mathcal{M}^2 G (K+1)
\sum_{h=0}^{k} U_h \mx{,} \\
\norm{(\widetilde{III})}{(\rho_{k}, \sigma_{k})} & \leq 
2K \mathcal{M} G 
\sum_{h=0}^{k} U_h \mx{,} \\
\norm{(\widetilde{IV})}{(\rho_{k}, \sigma_{k})} & \leq K \mathcal{M} G 
\sum_{h=0}^{j-2} U_h \leq K \mathcal{M} G 
\sum_{h=0}^{k} U_h \mx{,} \\
\|(\widetilde{V})\|_{(\rho_{k},\sigma_{k})}
& \leq 1 \mx{.} 
\end{align}
Hence, by collecting all the estimates, we have
\begin{align*}
\norm{\mathcal{E}_{N}^{[k]}}{(\rho_{k},\sigma_{k})} & \leq 
\left( \mathcal{M} G \left[\rho/\sigma + \mathcal{M}(K+1) + 3 K \right] \sum_{h=0}^{k} U_h\right) \left( \sum_{j=k+1}^{\infty} \mu^j \right) \\
& \leq \mu \mathcal{M} [\rho/\sigma+ 4 \mathcal{M}(K+1)] F(k) \mx{,}
\end{align*}
where we have used \eqref{eq:g}, \eqref{eq:boundsumul}, $\mathcal{M} \geq 1$, then $\sum_{j \geq k+1} \mu^j \leq 2 \mu^{k+1}$ for all $\mu \leq 1/2$ and finally the definition of $F(k)$ given in (\ref{eq:fk}). We proceed with the following bound
\begin{align*}
\mu\mathcal{M}\!\left[\rho/\sigma+4\mathcal{M}(K_{\mathrm{opt}}+1)\right]
&\le
\mu\left(
(\rho/\sigma)e^{\alpha}\mu^{-1/8}
+4e^{2\alpha}\mu^{-1/4}
\left(
2-(8\alpha)^{-1}\log\!\bigl(2e\mathscr{B}^2\mu\bigr)
\right)
\right)\\
&\le
e^{\frac{8}{\rho\sigma}}
\left(
(\rho/\sigma)\mu^{7/8}
+8\mu^{3/4}
+(e\rho/2)\mu^{1/2}
\right)\\
&\le
e^{\frac{8}{\rho\sigma}}
\left(\rho/\sigma+8+e\rho/2\right)\mu_0^{1/2}\\
&=
e^{\frac{8}{\rho\sigma}}
\left(\rho/\sigma+8+e\rho/2\right)\mathscr{C}_3^{-1}.
\end{align*}
Where we have used (\ref{eq:mkstar}), the inequalities $(K_{\text{opt}}+1) \leq (K_*+2)$ and $\mu^{\gamma}\log(1/(A\mu)) \leq 1/(e \gamma A)$, for all $\gamma,A >0$ (which can be easily shown by computing the maximum of the l.h.s. in $\mu$), then (\ref{eq:muzeronew}), the property $2 e \mathscr{B}^2 \ge 1$ and other immediate numerical bounds.
\[
\norm{\mathcal{E}_{N}^{[k]}}{(\rho_k,\sigma_k)}
\le e^{\frac{8}{\rho\sigma}}
\left(\rho/\sigma+8+e\rho/2\right)\mathscr{C}_3^{-1}
F(k) \mx{.} 
\]
By evaluating the latter in $k=K \equiv K_{\text{opt}}$, using once more the combination of (\ref{eq:prefinal}) and (\ref{eq:prefinalone}), then finally defining \beq{eq:cfive}
\mathscr{C}_5:=e^{\frac{8}{\rho\sigma}}
\left(\rho/\sigma+8+e\rho/2\right) \mathscr{C}_3^{-1} \mathscr{C}_1 \mx{,}
\eeq
the last bound of (\ref{eq:finalbound}) is proven. The proof is now complete. 
\clearpage
\section{A paradigmatic example}\label{sec:example}
The aim of this final section is to show how the formal algorithm presented in Sec. \ref{sec:formal} and summarised in pseudocode form in the Appendix, can be used in a concrete case to explicitly compute the approximants. For this reason, the computations will be carried out a bit more into detail for the reader's convenience and along the lines of the Appendix, to which we shall refer for all the formulae used at the various stages of the algorithm. \\  
For this purpose, let us consider the particularly simple perturbation function
\beq{eq:gex}
g(\theta) = (1/20) \cos( 4 \theta) \mx{.}
\eeq
\begin{rem} To complete the discussion started in Rem. \ref{rem:sigma} and continued in Rem. \ref{rem:rho}, we sketch how a particular choice of $g$ determines $\sigma$, completing in this way the pair of analyticity radii the quantitative bounds depend on. \\
We have $g'=-1/5 \sin (4 \theta)$, so that we find in this case
\[
\norm{g}{(\rho,\sigma)} < 
\norm{g'}{(\rho,\sigma)} = 1/5 \cosh(4 \sigma) \mx{.}
\]
Hence, requiring $\norm{g'}{(\rho,\sigma)} \leq 1/4$ is equivalent to
\[
\sigma \leq (1/4) \log 2 \sim 0.1733. 
\]
The latter unlocks the full determination of all the constants appearing in Thm. \ref{thm:main}, as well as the optimal normalisation order $K_{opt}$.  
\end{rem}
\subsection*{First step, $k=1$:}
\subsubsection*{Boundary conditions}
As $\mathcal{F}^{(1)} \equiv 0$ by construction, we shall start from the Dirichlet boundary condition. From  (\ref{eq:dirichlet}) we get that it reads as\[
u_1 (1,\theta) = - g(\theta) \partial_r u_0(r)=(1/40) \cos (4 \theta) \mx{,}
\]
so that the corresponding non-vanishing Fourier coefficients are reduced to $\widetilde{A}_4^{(1)} = 1/40$ only.\\
On the other hand, Neumann's first-order condition yields,
\[
\pl_r u_1 (1,\theta) = (1/2)g(\theta)=(1/40) \cos (4 \theta) \mx{,}
\]
implying that the only non-zero coefficient would be $\widetilde{A}_4^{(1)\prime}=1/40$.

\subsubsection*{Construction of $w_1$}
By choosing $M=0$, the construction of the sequence $\left\{ a_{n,0}^{(1)} \right\}_{n \geq 0} $ is reduced to 
\[
a_{4,0}^{(1)} = 2 (2+4)(4+4) \left(\widetilde{A}_4^{(1)\prime} - 4\widetilde{4}_2^{(1)}\right) = -36/5,
\]
while clearly $\left\{ b_{n,0}^{(1)} \right\}_{n \geq 1}$ vanish identically. By (\ref{eq:anseries}), one has $a_{4}^{(1)}=-36/5$, hence (\ref{eq:wkfourier}) readily yields
\beq{eq:examplew1}
w_1 = -(36/5) \cos (4 \theta).
\eeq
\subsubsection*{Construction of $u_1$}
First of all, by (\ref{eq:cndn}), 
\[
C_4^{(1)} = 2^{-1} \left( \widetilde{A}_4^{(1)\prime} - 4^{-1}\widetilde{A}_4^{(1)} \right) = 1/64, 
\]
while $D_n^{(1)} \equiv 0$. On the other hand, in this case we simply have $\mathfrak{f}_4^{(1)}(r)=u_0(r)a_{2}^{(1)}=-(36/20)(1-r^2)$. In conclusion, by (\ref{eq:an}), 
\[
A_{4}^{(1)}(r) =
r^4 \left(
C_4^{(1)} - \frac{1}{8}\int_r^1 \mathfrak{f}_4^{(1)} (s) \cdot s^{-3} ds
\right)
-\frac{1}{8 r^4}\int_0^r \mathfrak{f}_4^{(1)} (s) \cdot s^5 ds = \frac{9 r^4 \log r}{40}-\frac{r^4}{8} + \frac{3 r^2}{20}.
\]
Hence,
\beq{eq:exampleu1}
u_1(r,\theta) = (r^2/40)(9 r^2 \log r - 5 r^2 +6) \cos( 4 \theta).
\eeq
\subsection*{Second step, $k=2$:}
\subsubsection*{Non-linear term}
In this case we have a non-trivial non-linear term, which reads as
\[
\mathcal{F}^{(2)}:=w_1 u_1 = -(36/200)r^2(9 r^2 \log r - 5 r^2 +6) \cos^2( 4 \theta).
\]
As $ \cos^2( 4\theta) = (1/2)(1 + \cos (8\theta))$, this non-linear term gives contributions to the modes $0$ and $8$, more precisely 
\[
\xi_0^{(2)}(r) = \xi_8^{(2)}(r) =
-(36/100)r^2(9 r^2 \log r - 5 r^2 +6) ,
\]
while all the $\eta_k$ vanish identically.
\subsubsection*{Boundary conditions}
Let us now evaluate
\[
\Phi^{(2)}:=\left[-g \pl_r u_1-(1/2) g^2 \pl_r^2 u_0 \right]_{r=1}=-(1/1600)\cos^2(4\theta) \mx{,}
\]
leading to the following Fourier coefficients
\[
\widetilde{A}_0^{(2)} = \widetilde{A}_8^{(2)} = -1/3200, 
\]
while all the $\widetilde{B}_n^{(2)}$ vanish. As for the Neumann condition, we can refer to Rem. \ref{rem:firstorders}, case $k=2$, and evaluate
\[
\Psi^{(2)}=\left[g' \pl_{\theta} u_1 - g \pl_{r}^2 u_1 - 2 g \pl_r u_1 +g^2 -(1/4)(g')^2 \right]_{r=1}=(1/100) \sin^2 (4 \theta) - (3/160) \cos^2 (4 \theta) \mx{,}
\]
which directly leads to
\[
\widetilde{A}_0^{(2)\prime}=-7/1600, \qquad \widetilde{A}_8^{(2)\prime}=-23/1600 \mx{.}
\]
\subsubsection*{Construction of $w_2$}
It is sufficient to compute 
\begin{align*}
a_{0,0}^{(2)} &= 16 \left[ \widetilde{A}_0^{(2)} - \int_{0}^1 \xi_0^{(2)}(s) \cdot s ds   \right]=\frac{53}{100}, \\
a_{8,0}^{(2)} &= 2(2+8)(4+8) \left[ \widetilde{A}_8^{(2)} - 8 \widetilde{A}_8^{(2)\prime} - \int_{0}^1 \xi_8^{(2)}(s) \cdot s^9 ds   \right]=-\frac{741}{980},
\end{align*}
so that    
\beq{eq:examplew2}
w_2=53/100 - (741/980) \cos ( 8 \theta) \mx{,}
\eeq
furthermore, $a_{0}^{(2)}=53/100$ and $a_{8}^{(2)}=741/980$ immediately follow.
\subsubsection*{Construction of $u_2$}
The only non vanishing term of the sequences $\left\{ C_n^{(k)} \right\}_{n \geq 1} $ and $\left\{ D_n^{(k)} \right\}_{n \geq 1} $ is
\[
C_8^{(2)}=2^{-1} \left(\widetilde{A}_8^{(2)} + 8^{-1}  \widetilde{A}_8^{(2)\prime} \right) = -27/25600 \mx{.}
\]
On the other hand, 
\begin{align*}
\mathfrak{f}_0^{(2)}(r) &=u_0 a_{0}^{(2)} + \xi_0^{(2)}=(53/400) (1-r^2) -(36/100)r^2(9 r^2 \log r - 5 r^2 +6),\\
\mathfrak{f}_8^{(2)}(r) &=u_0 a_{8}^{(2)} + \xi_8^{(2)}=(741/3920) (1-r^2) -(36/100)r^2(9 r^2 \log r - 5 r^2 +6),
\end{align*}
implying, in particular, $\mathfrak{F}_0^{(2)}(s):=\int_{0}^s \mathfrak{f}_0^{(2)}(s') ds' =1600^{-1}s^2 (216 s^4 \log s - 156 s^4 + 269 s^2 -106)$.  As a consequence,  
\[
A_0^{(2)}= \widetilde{A}_0^{(2)} - \int_{r}^1 s^{-1} \mathfrak{F}_0^{(2)}(s)  ds=
-\frac{73}{6400}
+\frac{53}{1600} r^2
-\frac{269}{6400} r^4
+\frac{1}{50} r^6
-\frac{9}{400} r^6 \log(r)
 \mx{,}
\]
and finally
\begin{align*}
A_8^{(2)} & =r^8 \left[C_8^{(2)} - \frac{1}{16} \int_r^1 \mathfrak{f}_8^{(2)}(s) \cdot s^{-7} ds - \frac{1}{16 r^8} \int_0^r \mathfrak{f}_8^{(2)}(s) \cdot s^9 ds \right]\\
&=\frac{247}{78400} r^2
+\frac{2293}{313600} r^4
-\frac{9}{2450} r^6
+\frac{81}{2800} r^6 \log(r)
-\frac{2227}{313600} r^8 \mx{.}
\end{align*}
The computations above give the $u_2$ we were looking for
\beq{eq:exampleu2}
\begin{aligned}
u_2 & = \left( -\frac{73}{6400}
+\frac{53}{1600} r^2
-\frac{269}{6400} r^4
+\frac{1}{50} r^6
-\frac{9}{400} r^6 \log(r)
\right)\\
    & + \left( 
    \frac{247}{78400} r^2
+\frac{2293}{313600} r^4
-\frac{9}{2450} r^6
+\frac{81}{2800} r^6 \log(r)
-\frac{2227}{313600} r^8 
\right) \cos(8 \theta)\mx{.}
\end{aligned}
\eeq

\subsection*{Validation}
To give a more transparent validation, we shall test (\ref{eq:perturbedserrin}) directly, without passing through the hierarchy equations or other intermediate steps. As a result, we should check that the truncated solutions satisfy (\ref{eq:perturbedserrin}) up to $O(\mu^2)$, i.e. exactly, if $O(\mu^3)$ are disregarded. \\
For this purpose, we define
\[
u^{[2]}=u_0+\mu u_1 + \mu^2 u_2 \quad ; \quad w^{[2]}=w_1 + \mu w_2 \mx{,}
\]
where the involved functions are given by (\ref{eq:radial}), (\ref{eq:exampleu1}), (\ref{eq:exampleu2}), (\ref{eq:examplew1}), and (\ref{eq:examplew2}), respectively.
\subsubsection*{Equation}
To check the equation, it is immediate to compute the following objects 
\begin{align*}
-\Delta u^{[2]} &= 1
+\mu\left(
-\frac{9}{5}r^2
+\frac{9}{5}
\right)\cos(4\theta)
+\mu^2\Bigg[
\left(
\frac{81}{100}r^4\log(r)
-\frac{9}{20}r^4
+\frac{6879}{19600}r^2
+\frac{741}{3920}
\right)\cos(8\theta)\\
&\qquad\qquad
+\left(
\frac{81}{100}r^4\log(r)
-\frac{9}{20}r^4
+\frac{269}{400}r^2
-\frac{53}{400}
\right)
\Bigg]
\\
\mu w^{[2]} u^{[2]}&=\mu 
\left(
\frac{9}{5}r^2
-\frac{9}{5}
\right)\cos(4\theta)
+\mu^2\left[
\left(
\frac{741}{3920}r^2
-\frac{741}{3920}
\right)\cos(8\theta)\right.\\
&\left.
+\left(
-\frac{81}{50}r^4\log(r)
+\frac{9}{10}r^4
-\frac{27}{25}r^2
\right)\cos^{2}(4\theta)
-\frac{53}{400}r^2
+\frac{53}{400}
\right] +O(\mu^{3})
\end{align*}
It is sufficient to sum the terms above, use $\cos^2(4 \theta)=(1+\cos(8 \theta))/2$ in the latter and disregard $O(\mu^{3})$ to get the first of (\ref{eq:perturbedserrin}).
\subsubsection*{Dirichlet boundary condition}
This is a straightforward check. One just needs to evaluate $u^{[2]}(r,\theta)|_{r=1+\mu g}$ and perform a second order Taylor expansion (for instance via a symbolic manipulation software) to get
\[
u^{[2]}(1+\mu g(\theta),\theta)=O(\mu^3) \mx{.}
\]
The behaviour of the latter as $\mu$ varies is shown in Fig. \ref{fig:validation} (c).
\subsubsection*{Neumann boundary condition}
Let us start by computing the (inner) normal unit vector to the domain's boundary. In the basis $\{e_r,e_{\theta}\}$ the tangent vector to $\pl \Omega_{\mu}$ reads as $t=(\mu g', 1+ \mu g)=(-5^{-1} \mu \sin (4 \theta), 1+ 20^{-1} \mu \cos(4 \theta) )$, from which
\[
\nu=\left(400+40 \mu \cos(4\theta)+15\mu^2 \sin^2(4\theta)\right)^{-1/2}
\left(-20 -\mu\cos(4\theta), -4 \mu\sin(4\theta) \right) \mx{.}
\]
On the other hand, 
\begin{align*}
\nabla u^{[2]}=
\Bigg( &
-\frac{r}{2}
+\mu\left(
\frac{9r^{3}\log(r)}{10}
-\frac{11r^{3}}{40}
+\frac{3r}{10}
\right)\cos(4\theta)\\
&+\mu^{2}\left[
\left(
\frac{243r^{5}\log(r)}{1400}
-\frac{2227r^{7}}{39200}
+\frac{27r^{5}}{3920}
+\frac{2293r^{3}}{78400}
+\frac{247r}{39200}
\right)\cos(8\theta)\right.\\
&\left.
-\frac{27r^{5}\log(r)}{200}
+\frac{39r^{5}}{400}
-\frac{269r^{3}}{1600}
+\frac{53r}{800}
\right],\\
&-4\mu\left(
\frac{9r^{3}\log(r)}{40}
-\frac{r^{3}}{8}
+\frac{3r}{20}
\right)\sin(4\theta)\\
&-8\mu^{2}\left(
\frac{81r^{5}\log(r)}{2800}
-\frac{2227r^{7}}{313600}
-\frac{9r^{5}}{2450}
+\frac{2293r^{3}}{313600}
+\frac{247r}{78400}
\right)\sin(8\theta)\Bigg)
 \mx{.}
\end{align*}
Similarly to the case of the already tested Dirichlet condition, the property
\[
\pl_{\nu} u^{[2]} (1+\mu g, \theta):= \nabla u^{[2]} (1+\mu g, \theta) \cdot \nu = 1/2+ O(\mu^3) \mx{,}
\]
can be checked, once more, via a second order Taylor expansion. Its behaviour is shown in Fig. \ref{fig:validation} (d).\\
The validation is now complete.
\\
\begin{figure*}[t!]\begin{center}
\vspace{0pt}
\begin{minipage}[c][1\width]{0.48\textwidth}
\hspace{0pt}
	{\begin{overpic}[width=\textwidth]{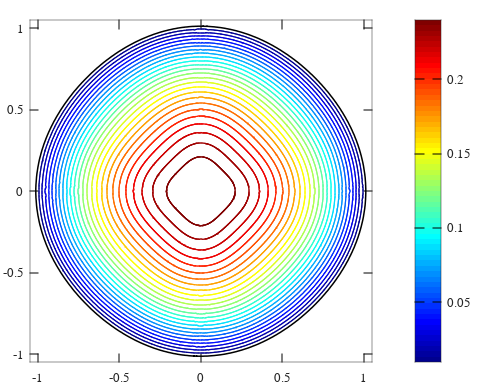}
		\put(00,82){\footnotesize (a)}
	\end{overpic}}
	\end{minipage}	
\begin{minipage}[c][1\width]{0.48\textwidth}
\hspace{10pt} 
	{\begin{overpic}[width=\textwidth]{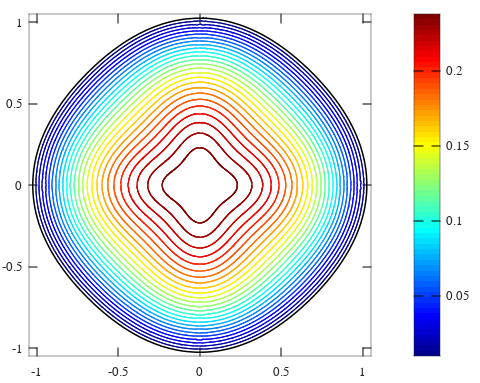}
	\end{overpic}}
	\end{minipage}\\
\vspace{-20pt}
\begin{minipage}[c][1\width]{0.48\textwidth}
\hspace{0pt}
	{\begin{overpic}[width=.95\textwidth]{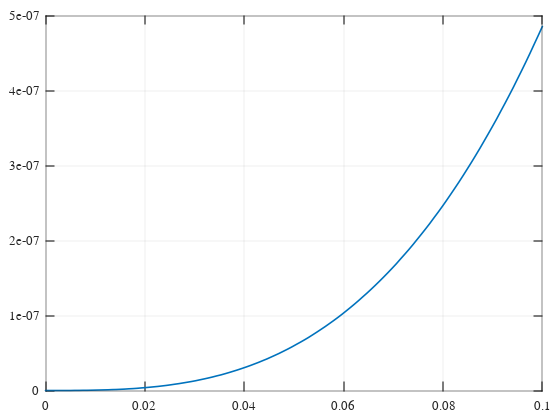}
		\put(0,80){\footnotesize (c)}
		\put(48,45){\footnotesize $\displaystyle\max_{\theta \in \mathbb{T}}\snorm{u^{[2]}|_{\pl \Omega_{\mu}}}{}$}
		\put(86,0){\footnotesize $\mu$}
	\end{overpic}}
	\end{minipage}	
\begin{minipage}[c][1\width]{0.48\textwidth}
\hspace{10pt} 
	{\begin{overpic}[width=.95\textwidth]{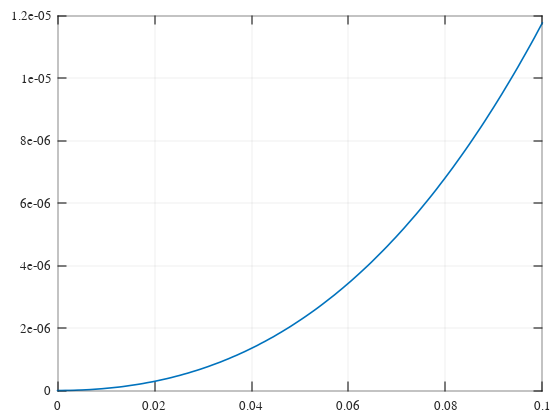}
		\put(0,80){\footnotesize (d)}
       	\put(33,45){\footnotesize $
        \displaystyle\max_{\theta \in \mathbb{T}}\snorm{     
        \pl_{\nu} u^{[2]}|_{\pl \Omega_{\mu}}}{}-\frac{1}{2}$}
  		\put(86,0){\footnotesize $\mu$}
	\end{overpic}}
	\end{minipage}
\end{center}
\vspace{-10pt}
\caption{In panels (a) and (b), the solution $u^{[2]}(r,\theta)$ plotted for $\mu=0.25$ and $\mu=0.5$, respectively. 
Panel (c) shows the plot of  $\max_{\theta \in \mathbb{T}}\snorm{u^{[2]} (1+\mu g(\theta), \theta)}{}$ as a function of $\mu$, while the behaviour of $\max_{\theta \in \mathbb{T}}\snorm{\pl_{\nu}  u^{[2]} (1+\mu g(\theta), \theta)}{}-1/2$ is shown in panel (d).}
\label{fig:validation}
\vspace{0pt}
\end{figure*}
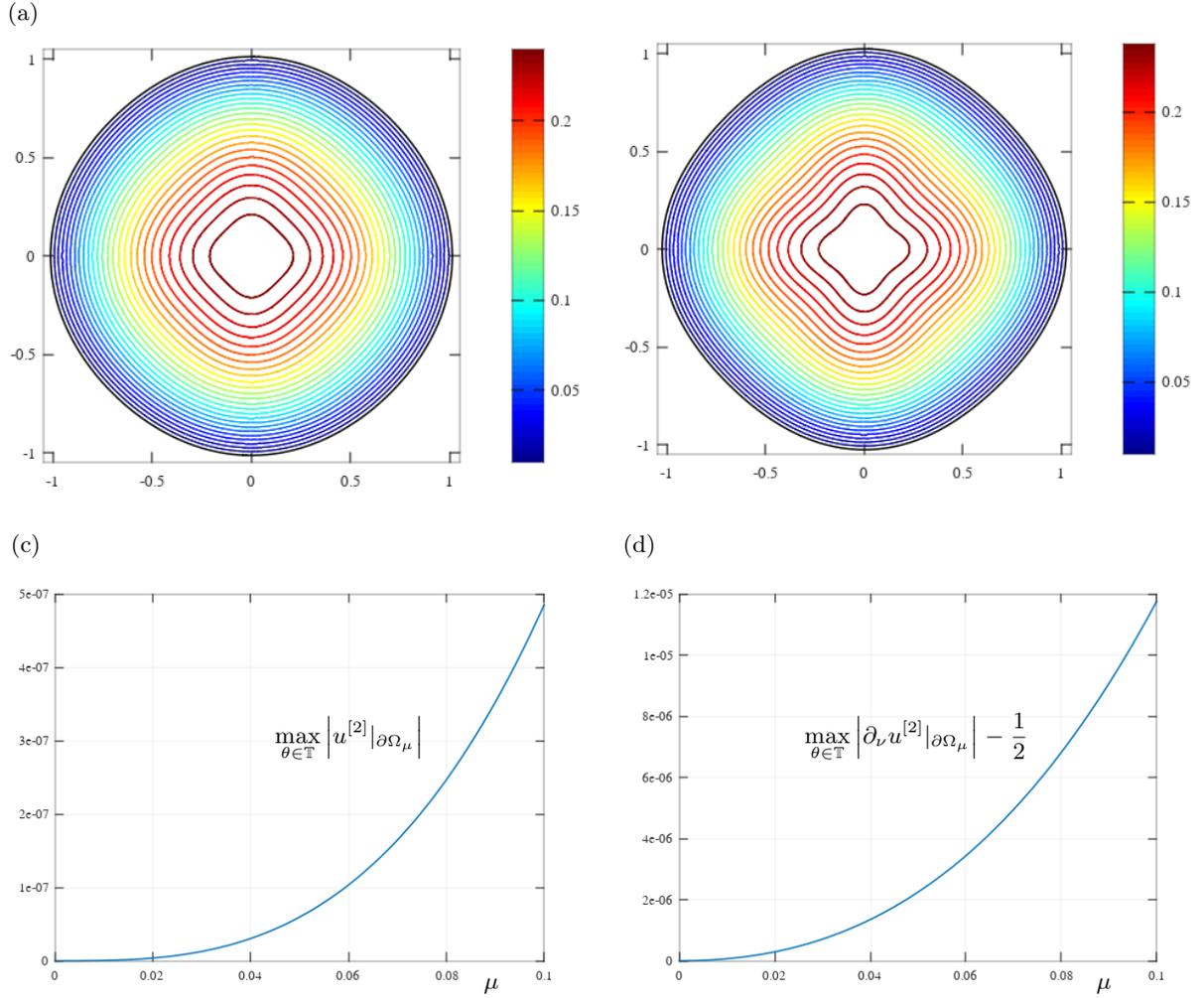

\clearpage
\section*{Appendix}\label{appendix:a}
The formal algorithm constructed in Sec. \ref{sec:formal} is presented here in the form of pseudocode. Given $g(\theta)$ as input, the $ k $-th order perturbative corrections $ u_k(r,\theta) $ and $ w_k(r,\theta) $ can be computed assuming that $ u_j(r,\theta) $ and  $ u_j(r,\theta) $ have been computed already, for all $ j < k $, as well as their derivatives at the boundary. Note that for $k=1$, $u_0(r)$ is known and given by (\ref{eq:radial}) and $w_0$ is identically zero by construction.
\begin{tcolorbox}[title=Algorithm: Recursive Computation at order $ k \geq 1 $, breakable, colback=white]
\begin{algorithmic}[1]
\State Compute the known term $\ml{F}^{(k)}(r, \theta)$ according to (\ref{eq:source}); 
\State Compute the Dirichlet boundary value $
u_k(1, \theta)$ via (\ref{eq:dirichlet}); 
\State Compute the Neumann boundary value $\partial_r u_k(1, \theta)$ via (\ref{eq:neumanneq});
\State Expand $
u_k(1, \theta)$, $\partial_r u_k(1, \theta)$ and $\ml{F}^{(k)}$ in Fourier series according to (\ref{eq:ukfourier}), (\ref{eq:expansionfourierdirichlet}) and (\ref{eq:fkfourier}), respectively, 
\begin{align*} 
u_k(1, \theta) & \rightarrow  \left\{ \widetilde{A}_n^{(k)} \right\}_{n \geq 0} \, ; \, \left\{ \widetilde{B}_n^{(k)} \right\}_{n \geq 1}, \\
\partial_r u_k(1, \theta) & \rightarrow  \left\{ \widetilde{A}_n^{(k)\prime} \right\}_{n \geq 0}  \, ; \, \left\{ \widetilde{B}_n^{(k)\prime} \right\}_{n \geq 1}, \\
\ml{F}^{(k)} & \rightarrow \left\{ \xi_n^{(k)}(r)  \right\}_{n \geq 0}  \, ; \, \left\{ \eta_n^{(k)}(r)  \right\}_{n \geq 1}
\end{align*}
and store the computed (non-vanishing) coefficients.
\State Set $M \geq 0$ and compute the $2 M$ free coefficients $a_{\cdot,m}^{(k)}$ and $b_{\cdot,m}^{(k)}$ satisfying (\ref{eq:decayanm}), then: 
\begin{itemize}
\item $a_{0,0}^{(k)}$ with (\ref{eq:azz}),
\item $\left\{ a_{n,0}^{(k)} \right\}_{n \geq 1} $ and $\left\{ b_{n,0}^{(k)} \right\}_{n \geq 1}$ by using (\ref{eq:anm}) and (\ref{eq:bnm}), respectively,
\end{itemize}
and finally define $\left\{ a_{n}^{(k)} \right\}_{n \geq 0} $ and $\left\{ b_{n}^{(k)} \right\}_{n \geq 1}$ via (\ref{eq:anseries}). 
\State Use the coefficients above to construct $w_k(r,\theta)$ according to (\ref{eq:wkfourier}).
\State Compute $\left\{ C_n^{(k)} \right\}_{n \geq 1} $ and $\left\{ D_n^{(k)} \right\}_{n \geq 1} $ via (\ref{eq:cndn}).
\State Define $\left\{ \mathfrak{f}_n^{(k)}(r) \right\}_{n \geq 0} $ by using (\ref{eq:ceazero}) and (\ref{eq:cean}), then $\left\{ \mathfrak{g}_n^{(k)}(r) \right\}_{n \geq 1}$ with (\ref{eq:cebn}).
\State Compute $\left\{  A_n^{(k)}(r) \right\}_{n \geq 0} $ according to (\ref{eq:azeropre}) and (\ref{eq:an}), then $\left\{  B_n^{(k)}(r) \right\}_{n \geq 1} $ via (\ref{eq:bn})
\State Use the coefficients above to construct $u_k(r,\theta)$ according to (\ref{eq:ukfourier}).
\end{algorithmic}
\end{tcolorbox}

\subsection*{Acknowledgements} 
The affiliation listed for the first author corresponds to the Institution where most of the work presented in this paper was carried out, although the author is no longer affiliated with it.\\
The plots reported have been performed with GNU Octave \cite{oct} while some symbolic computations were carried out with Maxima \cite{maxima}.
\clearpage
\bibliographystyle{alpha}
\bibliography{SerrinBib.bib}

\end{document}